\documentclass[reqno,11pt]{amsart}
\usepackage{cases}
\usepackage{mathrsfs}
\usepackage[T1]{fontenc}
\usepackage{mathrsfs}
\usepackage{amsmath,latexsym,amssymb,amsfonts,amsbsy, amsthm}
\usepackage{bm}
\usepackage[usenames]{color}
\usepackage{xspace,colortbl}
\usepackage{epsfig}
\usepackage{graphicx}
\usepackage{subfigure}
\usepackage{amsmath,amsfonts,amsthm,amssymb,amscd}
\input amssym.def 
\input amssym.tex
\usepackage{color}
\usepackage[title,titletoc,toc]{appendix}
\usepackage{bigints}
\allowdisplaybreaks[4]

\newtheorem {remark}{Remark}[section]
\newtheorem{theorem}{Theorem}[section]
\newtheorem{lemma}{Lemma}[section]
\newtheorem{definition}{Definition}[section]

\numberwithin{equation}{section}
\title{On a two-dimensional Camassa-Holm-Zakharov-Kuznetsov equation}

\author[Li]{Xixuan Li}
\address{Xixuan Li\newline
	School of Mathematics and Statistics, Huazhong University of Science and Technology, Wuhan, Hubei 430074, China}
\email{xixuanli@hust.edu.cn}

\author[Yan]
{Kai Yan\textsuperscript{*}}
\thanks{\noindent $^*$Corresponding author.}
\address{Kai Yan (Corresponding author) \newline
	School of Mathematics and Statistics, Huazhong University of Science and Technology, Wuhan, Hubei 430074, China}
\email{kaiyan@hust.edu.cn}

\author[Zhao]{Tiantian Zhao}
\address{Tiantian Zhao \newline
	School of Mathematics and Statistics, Huazhong University of Science and Technology, Wuhan, Hubei 430074, China}
\email{zhaotiantian@hust.edu.cn}

\begin{document}
	\begin{abstract}
		This paper is devoted to a new two-dimensional nonlinear dispersive wave model named as the Camassa-Holm-Zakharov-Kuznetsov (CH-ZK) equation which combines the nonlinear structure of the Camassa-Holm equation with the transverse Laplacian dispersion of the Zakharov-Kuznetsov equation.  We first establish the local well-posedness of its Cauchy problem in a suitable Sobolev space and derive a blow-up criterion for strong solutions. Then the finite-time blow-up strong solutions for the CH-ZK equation have been constructed. Moreover, we prove a unique continuation property for the solutions to the CH-ZK equation. Finally, we investigate the existence of both peaked and smooth solitary waves, and obtain a rigidity theorem for the traveling wave solutions according to the magnitude of wave speed.
	\end{abstract}
	
	\maketitle
	

	
	\noindent {\sl Keywords\/}:  Zakharov-Kuznetsov equation, Camassa-Holm equation,  Blow-up, Unique continuation property, Solitary wave.
	
	\vskip 0.2cm
	
	\noindent {\sl AMS Subject Classification} (2022): 35Q53; 35G25; 35B44; 76B25
	\renewcommand{\theequation}{\thesection.\arabic{equation}}
	\setcounter{equation}{0}

	\section{Introduction}
	The study of shallow water waves has attracted sustained attention over many decades. The Korteweg-de Vries (KdV) equation is a classical model in the study of water waves, which is completely integrable and admits smooth soliton solutions \cite{Korteweg-de Vries 1895}. It takes the form
	\begin{equation}\label{KdV equ}
		u_t + 6uu_x + u_{xxx} = 0.
	\end{equation}
	From the KdV equation, adding transverse effects leads to the following Kadomtsev-Petviashvili (KP-II) equation \cite{KP-70}
	\begin{equation}\label{KP-II equation}
		\left(  u_t + uu_x + u_{xxx} \right)_x + u_{yy} = 0 ,
	\end{equation}
	which is a two-dimensional generalization of the KdV equation. It is completely integrable and derived as a model for propagation of the weakly transverse water waves in a long wave regime \cite{KP-70}. The well-posedness of the KP-II equation has been studied extensively. For instance, it is globally well-posed in $L^2(\mathbb{R}^2)$ \cite{Bourgain GFA 1993}, in $H^s(\mathbb{R}\times\mathbb{T}) ~\text{for}~ s>0$ \cite{Molinet-Saut-Tzvetkov AIHCANL 2011}, and in the anisotropic homogeneous space $\dot{H}^{-\frac12,0}(\mathbb{R}^2)$ \cite{Hadac-Herr-Koch AIHCANL 2009}. Further discussions for the KP-II equation (\ref{KP-II equation}) may be found in the aforementioned papers and references therein.
	
	It is worth noting that the higher-dimensional generalizations of the KdV equation are not limited to the KP-II equation. Based on the fluid dynamics system, Zakharov and Kuznetsov systematically derived another important generalization, that is the following two dimensional Zakharov-Kuznetsov (ZK) equation \cite{Zakharov-Kuznetsov JETP 1974}
	\begin{equation}\label{ZK1}
		u_t+uu_x+u_{xxx}+u_{xyy}=0,
	\end{equation}
	as well as the higher dimensional version as follows:
	\begin{equation}\label{ZK2}
		u_t+\partial_{x_1}(\Delta u+u^2)=0,
	\end{equation}
	where \(u(t,x_1,y):[0,\infty)\times\mathbb{R}\times\mathbb{R}^{d-1}\to\mathbb{R}, ~\Delta:=\partial_{x_1}^2+\sum\limits_{j=1}^{d-1}\partial_{y_j}^2\), and $d$ is the spatial dimension. The ZK equation models the effect of a magnetic field on ion-acoustic waves in a plasma \cite{Linares-Panthee-Robert-Tzvetkov 2019,Seadawy 2014}. In recent years, the well-posedness theory for the ZK equations has been extensively studied. In the two-dimensional case, Faminskii \cite{Faminskii 1995} proved the global well-posedness of the Cauchy problem for the ZK equation in the energy space \( H^1(\mathbb{R}^2) \). Gr\"unrock and Herr \cite{GrunrockHerr 2014} as well as Molinet and Pilod \cite{Molinet-Pilod 2015} lowered the threshold to $H^s(\mathbb{R}^2)$ with \( s > \frac{1}{2} \). In higher dimensions, Ribaud and Vento proved its global well-posedness in \cite{Ribaud-Vento 2012 }. On the other hand, under suitable initial conditions, its blow-up solutions in finite-time have been constructed in \cite{Korpusov2014}.
	Besides, the unique continuation results of the ZK equation can be found in \cite{Bustamante-Isaza-Mejia JDE 2011,Bustamante-Isaza-Mejia JFA 2013,Panthee 2004}.
	
	Although both the KdV and KP equations capture weakly nonlinear and weakly dispersive effects, they fail to describe strongly nonlinear phenomena like wave breaking and peaked solitary waves. The  celebrated Camassa-Holm (CH) equation possesses such features and reads as follows:
	\begin{equation}\label{CH equation}
		u_t - u_{txx} + \kappa u_x + 3uu_x - (2u_xu_{xx} + uu_{xxx}) = 0,
	\end{equation}
	where \(\kappa\) is a real parameter, $u(t,x)$ denotes the fluid velocity at time $t$ and in the spatial $x$ direction. The CH equation (\ref{CH equation}) models the unidirectional propagation of shallow water waves over a flat bottom \cite{CH}, and also models the propagation of axially symmetric waves in hyperelastic rods \cite{Dai Act M 1998}, which possesses  a bi-Hamiltonian structure \cite{Fuchssteiner-Fokas PD 1981} and is completely integrable \cite{CH}. In contrast to the KdV equation (\ref{KdV equ}), the CH equation (\ref{CH equation}) captures strongly nonlinear phenomena, including wave breaking\cite{Constantin-Escher-acta} and peaked solitons \cite{CH}. The local well-posedness and blow-up properties of the CH equation were investigated by Constantin and Escher \cite{Constantin-Escher-AsNsP,Constantin-Escher-acta,Constantin-Escher-CPAM}. On the other hand, the global weak solutions  to the CH  equation is investigated in  \cite{Xin-Zhang CPAM 2000}. 
	
	In order to obtain a two-dimensional water wave model that captures strong nonlinearity and weak transverse effects, the CH equation and its two-dimensional extensions have attracted considerable attention. The Camassa-Holm-Kadomtsev-Petviashvili (CH-KP) equation has recently been derived from the three-dimensional Euler equations governing incompressible and irrotational flows, and it takes the form \cite{gui-liu-luo-yin-21}
	\begin{align} \label{CH-KP eqaution}
		\Bigl(u_t - u_{txx} + \kappa u_x + 3uu_x - (2u_xu_{xx} + uu_{xxx})\Bigr)_x + u_{yy} = 0,
	\end{align}
	where \(u = u(t,x,y)\) denotes the horizontal velocity field. The CH-KP \eqref{CH-KP eqaution} generalizes the CH equation (\ref{CH equation}) in the same manner as the KP-II equation (\ref{KP-II equation}) extends the classical KdV equation (\ref{KdV equ}). By incorporating both strong unidirectional dynamics and weak transverse modulations within a unified framework, the CH-KP model constitutes a powerful tool for the analysis of long-crested, moderately nonlinear shallow-water waves, and is directly relevant to geophysical phenomena such as tsunami evolution, tidal bores, and nearshore wave–current interactions. The formal Hamiltonian structure of the CH-KP (\ref{CH-KP eqaution}) was established in \cite{gui-liu-luo-yin-21}. Gui {\it et. al} \cite{gui-liu-luo-yin-21} also studied the local well-posedness in suitable Sobolev space and existence of traveling wave solutions. Later on, the transverse spectral stability of small periodic traveling waves of the equation (\ref{CH-KP eqaution}) was proved in \cite{Geyer-Liu-Pelinovsky JMPA 2024}. In  a related context, the transverse spectral stability of small periodic traveling waves of the $b$-KP equation as a two-dimensional generalization of the $b$-family of CH equation was demonstrated in \cite{Chen-Fan-Wang-Xu MA 2024}.  
	
	Since the KdV equation (\ref{KdV equ}) can be generalized to the KP equation (\ref{KP-II equation}) (by adding transverse derivative terms) as well as to the ZK equation (\ref{ZK1}) or (\ref{ZK2}) (by adding transverse Laplacian terms), and the CH equation (\ref{CH equation}) has already been successfully extended to the CH-KP equation (\ref{CH-KP eqaution}). This naturally raises the following 
	
	{\bf Question:} {\it Does there exist a two-dimensional Camassa-Holm equation that serves as the Camassa-Holm counterpart of the two-dimensional KdV equation (referred to as the Zakharov-Kuznetsov equation)? }

    The goal of the present paper is to study this kind of new model in two dimensions named as the Camassa-Holm-Zakharov-Kuznetsov (CH-ZK) equation which takes the following form
	\begin{equation}\label{ZK equation (1)}
		u_t - u_{xxt} + \kappa u_x + 3uu_x - (2u_xu_{xx}+uu_{xxx})+u_{xyy}=0,
	\end{equation}
	where \(u(t,x,y):[0,\infty)\times\mathbb{R}\times\mathbb{R}\to\mathbb{R}\) and \(\kappa \in \mathbb{R}\). It is not difficult to verify that the CH-ZK equation (\ref{ZK equation (1)}) has two conserved quantities, namely,
	\begin{equation}\label{ZK equation (2')}
		E(u):=\dfrac{1}{2}\int_{\mathbb{R}^2}\Bigl(u^2+u_x^2\Bigr)\,dxdy,\quad\quad\quad 
	\end{equation}
	and
	\begin{equation}\label{ZK equation (2'')}
		F(u):=\dfrac{1}{2}\int_{\mathbb{R}^2}\Bigl(\kappa u^2+u^3+uu_x^2-u^2_y\Bigr)\,dxdy.
	\end{equation}
	Define that
	$$
	J_1:=\partial_x(1-\partial_x^2),\quad\quad J_2:=\partial_x((m+\frac{\kappa}{2})\cdot)+(m+\frac{\kappa}{2})\partial_x+\partial_x\partial_y^2.
	$$
	A simple calculation then reveals that the CH-ZK equation (\ref{ZK equation (1)}) has the formal bi-Hamiltonian structure as follows:
	 $$ m_t = -J_1 \frac{\delta F}{\delta m} = -J_2 \frac{\delta E}{\delta m} ,$$
	 where \( J_1 \) and \( J_2 \) are two skew-symmetric differential operators.
	 
	For convenience, we rewrite (\ref{ZK equation (1)}) with the initial data $u_0$ as the following nonlocal weak form
	\begin{equation}\label{ZK equation (2)}
		\left\{ 
		\begin{aligned}
			&u_t+uu_x+G\ast\partial_x(u^2+\frac{1}{2}u_x^2+\kappa u)+G\ast u_{xyy}=0,\\
			& u(t,x,y)|_{t=0}=u_{0}\left(x,y\right),
		\end{aligned}
		\right.
	\end{equation}
	where \(G(x):=\frac{1}{2}e^{-|x|}\). By the structure of (\ref{ZK equation (2)}), we introduce the following Hilbert space $X^s$ for $s>0$:
	$$X^s=X^s\left(\mathbb{R}^2\right):=\left\{u\in H^s\left(\mathbb{R}^2\right)\mid u_x\in H^s\left(\mathbb{R}^2\right)\right\},
	$$
	equipped with the norm $\Vert u\Vert_{X^s\left(\mathbb{R}^2\right)}:=\left(\Vert u\Vert_{H^s\left(\mathbb{R}^2\right)}^{2}+\Vert u_x\Vert_{H^s\left(\mathbb{R}^2\right)}^{2}\right)^{\frac{1}{2}}$ and the inner product
	$$
	\left(  u, v\right)_{X^s} :=  (u,v)_{H^s} +  (u_x, v_x)_{H^s}, \quad \forall u, v\in X^s(\mathbb{R}^2).
	$$
	
	Now we are in  the position to state the local well-posedness and blow-up criterion for (\ref{ZK equation (2)}).
	\begin{theorem}\label{local ZK}
		Let $u_0 \in X^s(\mathbb{R}^2) $ with $s\geq2$. Then there exists a time $T>0$ such that the Cauchy problem (\ref{ZK equation (2)}) has a unique solution   \[u \in C([0,T];X^s(\mathbb{R}^{2}))\cap C^1([0,T];X^{s-2}(\mathbb{R}^{2}))\]  with the initial value $u_0$. Moreover, the solution $u  $ depends continuously on the initial value $u_0$. In addition, the conservation laws $E(u)$ and $F(u)$ in (\ref{ZK equation (2')})-(\ref{ZK equation (2'')}) are independent of the existence time $t>0$. 
	\end{theorem}
	\begin{theorem}[Blow-up criterion]\label{blow up criterion ZK}
		Let $u_0\in X^s(\mathbb{R}^2),s\geq2$ and $u$ be the corresponding solution to (\ref{ZK equation (2)}) as in Theorem \ref{local ZK}. Asumme $T^*$ is the maximal time of existence. Then
		\begin{equation}
			T^*<\infty\Longrightarrow\int_{0}^{T^*}\| \nabla u({\tau})\|_{L^{\infty}(\mathbb{R}^2)}\,d\tau=\infty.
		\end{equation}
	\end{theorem}
	\begin{remark}
		For the proof of Theorem \ref{blow up criterion ZK}, we carry out a detailed analysis depending on the regularity index $s$. More precisely,  by using the Sobolev embedding, we sequentially handle the cases $s=2$, $s=3$ and $s>3$, and the final intermediate case $2<s<3$ is then treated via the interpolation and density arguments. 
	\end{remark}
	
	\begin{remark}
		The local well-posedness result (Theorem \ref{local ZK}) and blow-up criterion (Theorem \ref{blow up criterion ZK}) also work for the Sobolev space $X^s(\Omega_1 \times \Omega_2)$, $s \geq 2$, where $\Omega_1$ and $\Omega_2$ are either the whole line $\mathbb{R}$ or the unit circle $\mathbb{T} = \mathbb{R}/\mathbb{Z}$.
	\end{remark}
	It is known that one of the more challenging aspects of the Cauchy problem for nonlinear evolution equation is to construct the  blow-up solutions in finite time. Here, we obtain the finite-time blow-up strong solutions of the Cauchy problem (\ref{ZK equation (2)}) under certain conditions.  Indeed, we track the dynamics of the effective blow-up quantity $B(t):= \int_{\mathbb{R}}  u_xdy$ at $x=0$. After a careful estimates, one gets a Riccati-type differential inequality for 
	\begin{align}
		\frac{d B}{dt} \leq -A_1 B^2(t) + A_2 \quad \text{with $A_1>0$ and $A_2 \geq 0$,} \nonumber 
	\end{align}
	which ensures finite-time blow-up. More precisely, we have the following result. 
	\begin{theorem}\label{blow-up phenomena}
		Fix $\phi \in H^2(\mathbb{R})$ with $\phi \geq 0$ and $\int_{\mathbb{R}} \phi \,dy=1$. Suppose that $u_0 \in X^s(\mathbb{R}^2)$ with $s\geq 2$. If  the corresponding solution $u$ to (\ref{ZK equation (2)}) satisfies that $u(t,x,y) =-u(t,-x,y)$ for all $ t \geq 0$ and $ (x,y) \in \mathbb{R}^2$, and the initial value $u_0$ satisfies
		\begin{align}
			\int_{\mathbb{R}} u_{0,x}(0,y) \phi(y) \,dy < 0, \nonumber 
		\end{align}
		then the corresponding strong solution $u$ blows up in finite time $T^{\ast}$ such that 
		\begin{equation}
			T^{\ast} \leq T_1 := -\frac{2}{\int_{\mathbb{R}}  u_{0,x} (0,y)\phi(y) dy}.\nonumber 
		\end{equation}
	\end{theorem}
	In the case $(x,y)\in \mathbb{R}\times \mathbb{T}$, we have a similar blow-up result as follows.
	\begin{remark}\label{blow-up phenomena (*)}
		Suppose that $u_0 \in X^s(\mathbb{R} \times \mathbb{T})$ with $s\geq 2$. Let $T^{\ast}$ be the maximal existence time of the corresponding solution $u$ to (\ref{ZK equation (2)}) with initial data $u_0$. If the solution $u$ satisfies that $u(t,x,y) =-u(t,-x,y)$ for all $ t\geq 0$ and $(x,y) \in \mathbb{R} \times \mathbb{T}$, and the initial data $u_0$ satisfies 
		\begin{align}
			\int_{\mathbb{T}} u_{0,x} (0,y) dy < 0, \nonumber 
		\end{align}
		then  the corresponding strong solution $u$ blows up in finite time $T^{\ast}$ such that 
		\begin{equation}
			T^{\ast} \leq T_2:= -\frac{2}{\int_{\mathbb{T}} u_{0,x}(0,y) dy} .\nonumber 
		\end{equation}
	\end{remark}
	Next, motivated by the argument in study of CH equation (\ref{CH equation}) in \cite{Linares & Ponce}, we obtain the following unique continuation property for CH-ZK equation (\ref{ZK equation (2)}).
	\begin{theorem}\label{Liuville solution}
		Assume $\kappa=0$. Let $u=u(t,x,y)$ be a non-trivial solution of (\ref{ZK equation (2)}) in $ C([0,T];X^s(\mathbb{R}^{2}))\cap C^1([0,T];X^{s-2}(\mathbb{R}^{2})), s\geq 2$ with the lifespan T. Then there is no any open set $\Omega\subset[0,T]\times\mathbb{R}$ such that $u(t,x,y)=0,\forall (t,x)\in \Omega,y\in\mathbb{R}.$
	\end{theorem}
	The properties of localized traveling waves play a crucial role in the study of nonlinear dispersive equations. By assuming a peaked solution of the form $u=c e^{-|x+\beta y-ct|}$ and computing the convolution terms in the CH-ZK equation \eqref{ZK equation (2)} in a piecewise manner, we derive the necessary and sufficient condition $\kappa+\beta^2=0$ for the existence of such peakons.  On the other hand, for the smooth solitary waves, by introducing an auxiliary function $\psi$, we obtain the explicit relation between $\xi:=x+\beta y-ct$ and $\psi$, which is strictly monotone and smooth. 
	This guarantees the existence of smooth solitary waves precisely when $\kappa+\beta^2\neq0$. So, these two complementary cases lead to the following theorem which provides a complete classification of traveling solitary waves for the CH-ZK equation.
	\begin{theorem}\label{xingbojie}
		For the CH-ZK equation (\ref{ZK equation (2)}), the following classification holds for traveling-wave solutions of the form
		\[
		u(t,x,y)=\phi(\xi),\qquad \xi=x+\beta y-ct,
		\]
		with $\beta,c\in\mathbb{R}$, and $\phi\to0$ as $|\xi|\to\infty$.
		\begin{enumerate}
			\item[(1)] 
			There exists a global weak solution in the peak form
			\[
			u(t,x,y)=c\,e^{-|x+\beta y-ct|}
			\]
			if and only if
			$
			\kappa+\beta^2=0.
			$
			\item[(2)] 
			There exists a localized smooth solitary-wave solution if and only if
			$\kappa+\beta^2\neq0.$
		\end{enumerate}
	\end{theorem}
	Finally, let us consider the traveling wave solutions of the form $u(x-ct,y)$, where $c\in \mathbb{R}$ denotes the wave speed, then such solutions satisfy the following equation:
	\begin{equation}\label{ZK equation of travelling wave solutions-1}
		\left((\kappa-c)u+cu_{xx}+\frac{3}{2}u^2-\frac{1}{2}u^2_x-uu_{xx}\right)_x+u_{xyy}=0.
	\end{equation} 
	By studying the structure of the equation (\ref{ZK equation of travelling wave solutions-1}), the following theorem demonstrates the behavior of the solutions according to the magnitude of the wave speed $c$. 
	\begin{theorem}\label{Rigidity of solitary waves}
		Let $u(t,x,y)=u(x-ct,y)$ be the solutions to the equation (\ref{ZK equation of travelling wave solutions-1}).
		\begin{enumerate}
			\item[(1)] 
			If $c<\min\left\{0,\kappa\right\}$, then $u\equiv0.$
			\item[(2)] 
			If $c>\max\left\{0,\kappa\right\}$, then u must be symmetric in x-variable.
		\end{enumerate}
	\end{theorem}

	The remainder of the paper is organized as follows. In Section 2, we establish the local well-posedness  (Theorem \ref{local ZK}) and derive the blow-up criterion (Theorems \ref{blow up criterion ZK}) for the Cauchy problem (\ref{ZK equation (2)}). Section 3 is devoted to  constructing  the finite-time blow-up solutions (Theorem \ref{blow-up phenomena}). In Section 4, we present the unique continuation property (Theorem \ref{Liuville solution}) of the solutions to the CH-ZK equation (\ref{ZK equation (2)}). Finally, in Section 5, by proving Theorems \ref{xingbojie}-\ref{Rigidity of solitary waves},  we describe the traveling-wave solutions to the CH-ZK equation and discuss its symmetry. 
	
	\section{Local well-posedness and blow-up criterion} 
	In this section, we consider the issue of local well-posedness (Theorem \ref{local ZK}) and blow-up criterion(Theorem \ref{blow up criterion ZK}) for the Cauchy problem (\ref{ZK equation (2)}). In order to prove Theorem \ref{local ZK}, we need the following useful lemmas.  
	\begin{lemma}[\cite{Moser J}]\label{lemma-local well-posedness-1}
		Let $\Lambda ^s :=(1-\Delta_{x,y})^{\frac{s}{2}}$ with $s>0,p_1,p_2\in [2,+\infty),q_1,q_2\in (2,+\infty],\frac{1}{p_1}+\frac{1}{q_1}=\frac{1}{p_2}+\frac{1}{q_2}=\frac{1}{2}$. Then there holds
		\begin{equation}\label{lemma-local well-posedness-1-ineq}
			\|fg\|_{H^s(\mathbb{R}^2)}\leq C\left( \|\Lambda^s f\|_{L^{p_1}(\mathbb{R}^2)} \|g\|_{L^{q_1}(\mathbb{R}^2)} + \|\Lambda^s g\|_{L^{p_2}(\mathbb{R}^2)} \|f\|_{L^{q_2}(\mathbb{R}^2)} \right),
		\end{equation}
		where the constant \( C \) is independent of \( f \) and \( g \).
	\end{lemma}
	\begin{lemma}[\cite{KaTo T-Pomce G}]\label{lemma-local well-posedness-2}
		Let \( \Lambda^s := (1 - \Delta_{x,y})^{\frac{s}{2}} \) with \( s > 0 \). Then the following commutator estimates hold ture:
		\begin{enumerate}
			\item[(i)] $\|[\Lambda^s, f]g\|_{L^2(\mathbb{R}^2)} \leq C \left( \|\Lambda^s f\|_{L^2(\mathbb{R}^2)} \|g\|_{L^\infty(\mathbb{R}^2)} + \|\nabla f\|_{L^\infty(\mathbb{R}^2)} \|\Lambda^{s-1}g\|_{L^2(\mathbb{R}^2)} \right);$
			\item[(ii)] $ \|[\Lambda^s, f]g\|_{L^2(\mathbb{R}^2)} \leq C \|\nabla f\|_{H^{q_0}(\mathbb{R}^2)} \|g\|_{H^{s-1}(\mathbb{R}^2)}, ~ \forall q_0 > 1, \, 0 \leq s \leq q_0 + 1,$
		\end{enumerate}
		where all the constants \( C \) are independent of \( f \) and \( g \).
	\end{lemma}
	\begin{lemma}[\cite{gui-liu-luo-yin-21}]   \label{lemma-local well-posedness-3}
		For $s\geq 2$, the space $X^s(\mathbb{R}^2)$ is continuously embedded in $Lip(\mathbb{R}^2)$.
	\end{lemma}
	\begin{lemma}[\cite{gui-liu-luo-yin-21}] 
		For all $u\in X^2(\mathbb{R}^2)$, we have
		\begin{equation}\label{remark-lemma-local well-posedness-3}
			\|u\|_{L^{\infty}}\leq C(\|u\|_{L^2}+\|u_x\|_{L^2}+\|u_y\|_{L^{\infty}}).
		\end{equation}
	\end{lemma}
	\begin{lemma}\label{Sobolev embedding}
		Let $ u, \, v \in H^1(\mathbb{R}^2)$, the following  inequality holds  true
		\begin{align}
			\Vert uv \Vert_{L^2(\mathbb{R}^2)} \leq \Vert u \Vert_{L^2(\mathbb{R}^2)}^{\frac{1}{2}} \Vert u_x \Vert_{L^2(\mathbb{R}^2)}^{\frac{1}{2}}\Vert v \Vert_{L^2(\mathbb{R}^2)}^{\frac{1}{2}}
			\Vert v_y \Vert_{L^2(\mathbb{R}^2)}^{\frac{1}{2}}. 
			\label{Sobolev embedding remark-2 eq}
		\end{align}
	\end{lemma}
	\begin{proof}
		For all $u, \, v\in H^1(\mathbb{R}^2)$, the following estimate  holds:
		\begin{align}
			\Vert uv \Vert_{L^2} \leq \Vert u \Vert_{L_x^{\infty} L_y^2} \Vert v \Vert_{L_x^2 L_y^{\infty}}. 
			\label{Sobolev embedding remark-2 proof eq}
		\end{align}
		Thanks to $u , \, v\in H^1(\mathbb{R}^2)$, we get  \[ u^2(x,y) = 2 \int_{-\infty}^x u(x',y)u_x (x' ,y) dx', \quad v^2(x,y) =2 \int_{-\infty}^y v (x,y') v_y(x,y') dy'. \] 
		Thus,  by the Cauchy-Schwartz inequality, we have 
		\begin{align}
			\Vert u \Vert_{L_x^{\infty}}^2 \leq 2  \Vert u \Vert_{L_x^2} \Vert u_x \Vert_{L_x^2},  \quad 
			\Vert v\Vert_{L_x^2}^2  \leq  2 \Vert v \Vert_{L^2_{x,y} } \Vert v_y \Vert_{L_{x,y}^2},\nonumber 
		\end{align}
		which along with (\ref{Sobolev embedding remark-2 proof eq})   completes the proof of Lemma \ref{Sobolev embedding}.
	\end{proof}
	\begin{proof}[Proof of Theorem \ref{local ZK}]
		Let us firstly derive some necessary a priori estimates. For this, applying the operator \(\Lambda^s\) to both sides of (\ref{ZK equation (1)}), one gets
		\begin{equation}\label{ZK equation (3)}
			\partial_t\Lambda^su-\partial_t\Lambda^su_{xx}+3\Lambda^suu_{x}-2\Lambda^su_{x}u_{xx}-\Lambda^suu_{xxx}+\Lambda^su_{xyy}+\kappa\Lambda^su_x=0.
		\end{equation}
		Note that
		$$
		\int_{\mathbb{R}^2}\Lambda^suu_{xxx}\Lambda^su~dxdy =\int_{\mathbb{R}^2}\Lambda^s\partial_{x}\left(uu_{xx}-\frac{1}{2}(u_{x})^2\right)\Lambda^su~dxdy.
		$$
		Multiplying (\ref{ZK equation (3)}) by \(\Lambda^s u\) and integrating it over $\mathbb{R}^2$, one infers
		\begin{equation}\label{ZK equation (5)^*}
			\begin{split}
				\frac{1}{2}\frac{d}{dt}\left(\Vert\Lambda^s\partial_{x}u\Vert_{L^2}^2+\Vert\Lambda^su\Vert_{L^2}^2\right)=\int_{\mathbb{R}^2}-\left(\Lambda^su_{xyy}+\kappa\Lambda^su_x \right)\Lambda^su~dxdy
				\\
				+\int_{\mathbb{R}^2}\left(-\Lambda^suu_{xx}-\frac{1}{2}\Lambda^s(u_{x})^2+\frac{3}{2}\Lambda^su^2 \right)\Lambda^s\partial_{x}u~dxdy .
			\end{split}
		\end{equation}
		Thanks to integration by parts, we have
		\[
		\left|-\int_{\mathbb{R}^2}\left(\Lambda^su_{xyy}+\kappa\Lambda^su_x\right)\Lambda^su\,dxdy\right|=0.
		\]
		For the remaining terms, by the Cauchy-Schwarz inequality and Lemma \ref{lemma-local well-posedness-1}, one can see
		\begin{equation}\label{lwp-1}
			\left| \int_{\mathbb{R}^2} \Lambda^s(u_{x})^2  \cdot\partial_{x} \Lambda^s u ~dx \right| \leq C \| u_{x}^2 \|_{H^s} \| \partial_{x} u \|_{H^s}\leq C \|u_{x}\|_{L^{\infty}}\| u_{x} \|_{H^s} \| \partial_{x}  u \|_{H^s},
		\end{equation}
		and
		\begin{equation}\label{lwp-2}
			\left| \int_{\mathbb{R}^2} \Lambda^s(u^2)  \cdot\partial_{x} \Lambda^s u ~dx \right| \leq C \| u^2 \|_{H^s} \| \partial_{x}  u \|_{H^s}\leq C \|u\|_{L^{\infty}}\| u \|_{H^s} \| \partial_{x}  u \|_{H^s}.
		\end{equation}
		Simultaneously, in view of Lemma \ref{lemma-local well-posedness-2}, we deduce
		\begin{align}
			\left|-\int_{\mathbb{R}^2}\Lambda^suu_{xx}\cdot\Lambda^su_{x}~dxdy\right|
			&=\left|\frac{1}{2}\int_{\mathbb{R}^2}u_{x}(\Lambda^su_{x})^2~dxdy-\int_{\mathbb{R}^2}\left[\Lambda^s,u\right]u_{xx}\cdot\Lambda^su_{x}~dxdy\right|\nonumber\\
			&\leq C\left(\|u_{x}\|_{L^{\infty}} \|u_{x}\|_{H^s}^2+\|\left[\Lambda^s,u\right]u_{xx}\|_{L^2}\|u_{x}\|_{H^s}\right)\nonumber\\
			&\leq C\left(\|u_{x}\|_{L^{\infty}} \|u_{x}\|_{H^s}+\|\nabla u\|_{L^{\infty}} \|u_{xx}\|_{H^{s-1}}+\|u_{xx}\|_{L^{\infty}} \|u\|_{H^s}\right)\|u_{x}\|_{H^s}.\label{lwp-3}
		\end{align}
		Plugging (\ref{lwp-1}), (\ref{lwp-2}) and (\ref{lwp-3}) into (\ref{ZK equation (5)^*}) gives rise to
		\begin{equation}\label{xian yan gu ji 1}
			\frac{d}{dt} \left( \| \partial_{x}u \|_{H^s}^2 + \| u \|_{H^s}^2 \right) \leq C\|u_{x}\|_{H^s}
			\left(\|u \|_{L^{\infty}}\|u \|_{H^s}+\| \nabla u\|_{L^{\infty}}\|u_{x} \|_{H^{s}}+\| u_{xx}\|_{L^{\infty}}\|u \|_{H^s}\right).
		\end{equation}
		By means of the Sobolev imbedding inequality, we have
		\begin{equation}\label{xian yan gu ji jie guo 1}
			\frac{d}{dt} \left( \| \partial_{x}u \|_{H^s}^2 + \| u \|_{H^s}^2 \right) \leq C \|u\|_{X^s}^3,~\text{when}~ s> 2.
		\end{equation}
		
		For the case $s=2$, we act the operators $\nabla$, $ \nabla \partial_x$ and $\nabla \partial_x^2$ to (\ref{ZK equation (2)}) respectively, which gives rise to 
		\begin{align}
			&\partial_t \nabla u   + u_x \nabla u  + u \nabla u_x +G\ast(\nabla u_{xyy}+\kappa\nabla u_x)+G_x\ast(2u\nabla u+u_x\nabla u_x) =0 , \label{case s=2 eq-1} \\
			& \partial_t \nabla u_x  + u_{xx} \nabla u + u_x \nabla u_x + u\nabla u_{xx}+G_x\ast(\nabla u_{xyy}+\kappa\nabla u_x)  \nonumber \\
			&\quad\quad\quad\quad\quad\quad\quad\quad\quad+G\ast(2u\cdotp\nabla u+u_x\nabla u_x)  -2u\nabla u =0 , \label{case s=2 eq-2} \\
			& \partial_t \nabla u_{xx} + u_{xxx} \nabla u + 2 u_{xx} \nabla u_x + 2 u_x \nabla u_{xx} + u \nabla u_{xxx} +G_{xx}\ast(\nabla u_{xyy}+\kappa\nabla u_x)   \nonumber \\
			&\quad\quad\quad\quad\quad\quad\quad\quad\quad+G_x\ast(2u\nabla u+u_x\nabla u_x) -2u_x\nabla u-2u\nabla u_x =0 . \label{case s=2 eq-3} 
		\end{align}
		Taking the $L^2$ inner-product between three equations in (\ref{case s=2 eq-1})-(\ref{case s=2 eq-3}) and $\nabla u$, $\nabla u_x$, $\nabla u_{xx}$ respectively, and summating them, we get
		\begin{align}
			&\frac{1}{2}\frac{d}{dt}(\|\nabla u\|^2_{L^2}+\|\nabla u_x\|^2_{L^2})=I_1+I_3+I_4+I_6,\label{case s=2 fanshuguji-a}\\
			&\frac{1}{2}\frac{d}{dt}(\|\nabla u\|^2_{L^2}+\|\nabla u_x\|^2_{L^2}+\|\nabla u_{xx}\|^2_{L^2})=\sum_{i=1}^{7}I_i,\label{case s=2 fanshuguji-a^*}
		\end{align}
		where
		\begin{align*}
			&I_1:=-\int_{\mathbb{R}^2}(u\cdotp\nabla u_x\cdotp\nabla u+u\cdotp\nabla u_{xx}\cdotp\nabla u_x)~dxdy,\\
			&I_2:=-\int_{\mathbb{R}^2}u\cdotp\nabla u_{xxx}\cdotp\nabla u_{xx}~dxdy,\\
			&I_3:=-\int_{\mathbb{R}^2}\Bigl(u_x\nabla u+G_x\ast(2u\nabla u+u_x\nabla u_x)\Bigr)\nabla u~dxdy,\\
			&I_4:=-\int_{\mathbb{R}^2}\Bigl(u_x\nabla u_x+u_{xx}\nabla u-2u\nabla u+G\ast(2u\nabla u+u_x\nabla u_x)\Bigr)\nabla u~dxdy,\\
			&I_5:=-\int_{\mathbb{R}^2}\Bigl(2u_x\nabla u_{xx}+2u_{xx}\nabla u_x+u_{xxx}\nabla u-2u\nabla u_x-2u_{x}\nabla u\\
			&\quad\quad\quad\quad\quad\quad+G_x\ast (2u\nabla u+u_x\nabla u_x)\Bigr)\nabla u_{xx}~dxdy,\\
			&I_6:=-\int_{\mathbb{R}^2}G\ast(\nabla u_{xyy}+\kappa\nabla u_x)\cdot \nabla u-G_x\ast(\nabla u_{xyy}+\kappa\nabla u_x)\cdot \nabla u_x~dxdy,\\
			&I_7:=-\int_{\mathbb{R}^2}G_{xx}\ast(\nabla u_{xyy}+\kappa\nabla u_x)\cdotp\nabla u_{xx}~dxdy.
		\end{align*}
		Thanks to integration by parts,we get $I_6=0,I_7=0$,and
		\begin{align}
			&|I_1|=\left|\frac{1}{2}\int_{\mathbb{R}^2}u_x(|\nabla u|^2+|\nabla u_x|^2)dxdy\right|\leq C\|u_x\|_{L^{\infty}}(\|\nabla u\|_{L^2}^2+\|\nabla u_x\|_{L^2}^2),\label{case s=2 fanshuguji-I1}\\
			&|I_2|=\left|\frac{1}{2}\int_{\mathbb{R}^2}u_x|\nabla u_{xx}|^2dxdy\right|
			\leq C\|u_x\|_{L^{\infty}}\|\nabla u_{xx}\|_{L^2}^2.\label{case s=2 fanshuguji-I2}
		\end{align}
		It is noticed that
		\begin{align}
			|I_3|&\leq C(\|\nabla u\|_{L^{\infty}}\|u_x\|_{L^2}+\|2u\nabla u+u_x\nabla u_x\|_{L^2})\|\nabla u\|_{L^2}\nonumber\\
			&\leq C\|\nabla u\|_{L^{\infty}}(\|\nabla u\|_{L^2}^2+\|\nabla u_x\|_{L^2}^2),\label{case s=2 fanshuguji-I3}
		\end{align}
		\begin{align}
			|I_4|&\leq C(\|u_x\|_{L^{\infty}}\|\nabla u_x\|_{L^2}+\|u_{xx}\|_{L^2}\|\nabla u\|_{L^{\infty}}+\|u\|_{L^2}\|\nabla u\|_{L^{\infty}})\|\nabla u_x\|_{L^2}\nonumber\\
			&\leq C\|\nabla u\|_{L^{\infty}}(\|u\|_{L^2}^2+\|\nabla u_x\|_{L^2}^2),\label{case s=2 fanshuguji-I4}
		\end{align}
		and
		\begin{align}
			|I_5|&\leq C(\|u_x\|_{L^{\infty}}\|\nabla u_{xx}\|_{L^2}+\|u_{xx}\|_{L^2_xL^{\infty}_y}\|\nabla u_x\|_{L^{\infty}_xL^2_y}+\|u_{xxx}\|_{L^2}\|\nabla u\|_{L^{\infty}} \nonumber\\
			&+\|u\|_{L^{\infty}}\|\nabla u_x\|_{L^2}+\|u_x\|_{L^2}\|\nabla u\|_{L^{\infty}}+\|u_x\|_{L^{\infty}}\|\nabla u_x\|_{L^2})\|\nabla u_{xx}\|_{L^2}.\label{case s=2 fanshuguji-I5}
		\end{align}
		Thanks to the Lemma \ref{Sobolev embedding}, we have
		$$
		\|u_{xx}\|_{L^2_xL^{\infty}_y}\|\nabla u_x\|_{L^{\infty}_xL^2_y}\leq C\|u_{xx}\|_{L^2}^{\frac{1}{2}}\|u_{xxy}\|_{L^2}^{\frac{1}{2}}\|\nabla u_x\|_{L^2}^{\frac{1}{2}}\|\nabla u_{xx}\|_{L^2}^{\frac{1}{2}},
		$$
		which along with (\ref{case s=2 fanshuguji-I5}) implies
		\begin{equation}\label{case s=2 fanshuguji-I5^*}
			|I_5|\leq C(\|u\|_{L^{\infty}}+\|\nabla u\|_{L^{\infty}}+\|\nabla u_x\|_{L^2})(\|\nabla u_{xx}\|^2_{L^2}+\|\nabla u_{x}\|^2_{L^2}+\| u_{x}\|^2_{L^2}).
		\end{equation}
		Hence, we obtain
		\begin{equation}\label{case s=2 fanshuguji-a-jieguo}
			\frac{d}{dt}(\|\nabla u\|^2_{L^2}+\|\nabla u_x\|^2_{L^2})\leq C\|\nabla u\|_{L^{\infty}}(\|\nabla u\|^2_{L^2}+\|\nabla u_x\|^2_{L^2}),
		\end{equation}
		and
		\begin{align}
			&\frac{d}{dt}(\|\nabla u\|^2_{L^2}+\|\nabla u_x\|^2_{L^2}+\|\nabla u_{xx}\|^2_{L^2})\nonumber\\
			&\leq C(\|u\|_{L^{\infty}}+\|\nabla u\|_{L^{\infty}}+\|\nabla u_x\|_{L^2})(\|\nabla u_{xx}\|^2_{L^2}+\|\nabla u_{x}\|^2_{L^2}+\| u_{x}\|^2_{L^2}).\label{case s=2 fanshuguji-a^*-jieguo}
		\end{align}
		While acting the operators $\partial^2_y$ and $\partial_x\partial^2_y$ to  (\ref{ZK equation (2)}) respectively gives rise to
		\begin{align}
			& \partial_tu_{yy}+uu_{xyy}+2u_yu_{xy}+u_xu_{yy}+G\ast(u_{xyyyy}+\kappa u_{xyy}) \label{case s=2 eq-4} \\
			&\quad\quad\quad+G_x\ast(2uu_{yy}+2u^2_y+u_xu_{xyy}+u^2_{xy})=0,\nonumber\\
			&\partial_tu_{xyy}+uu_{xxyy}+u_xu_{xyy}+u_{xy}^2+2u_yu_{xxy}+u_{xx}u_{yy}+G_x\ast(u_{xyyyy}+\kappa u_{xyy}) \label{case s=2 eq-5} \\
			&\quad\quad\quad+G\ast(2uu_{yy}+2u^2_y+u_xu_{xyy}+u^2_{xy})-2uu_{yy}-2u^2_y=0.\nonumber
		\end{align}
		Taking the $L^2$ inner-products between two equations in (\ref{case s=2 eq-4})-(\ref{case s=2 eq-5}) and $u_{yy}$, $u_{xyy}$ respectively, and summating them, we get
		\begin{equation}\label{case s=2 fanshuguji-b}
			\frac{1}{2}\frac{d}{dt}(\|u_{yy}\|^2_{L^2}+\|u_{xyy}\|^2_{L^2})=J_1+J_2+J_3+J_4,
		\end{equation}
		where
		\begin{align*}
			&J_1:=-\int_{\mathbb{R}^2}(uu_{xxyy}u_{xyy}+uu_{xyy}u_{yy})~dxdy,\\
			&J_2:=-\int_{\mathbb{R}^2}\Bigl((u_xu_{xyy}+u^2_{xy}+2u_yu_{xxy}+u_{yy}u_{xx}-2u^2_y-2uu_{yy})u_{xyy}\\
			&\quad\quad\quad\quad\quad+(2u_{y}u_{xy}+u_{x}u_{yy})u_{yy}\Bigl)~dxdy,\\
			&J_3:=-\int_{\mathbb{R}^2}\Bigl(G_x\ast(2uu_{yy}+2u^2_y+u_xu_{xyy}+u^2_{xy})u_{yy}  \\
			&\quad\quad\quad\quad\quad+G\ast(2u^2_y+2uu_{yy}+u^2_{xy}+u_xu_{xyy})u_{xyy}\Bigl)~dxdy,\\
			&J_4:=-\int_{\mathbb{R}^2}G\ast(u_{xyyyy}+\kappa u_{xyy})\cdotp u_{yy}+G_x\ast(u_{xyyyy}+\kappa u_{xyy})\cdotp u_{xyy}  ~dxdy.
		\end{align*}
		Applying intergration by part, one has $J_3=J_4=0$ and
		\begin{align*}
			&J_1=\frac{1}{2}\int_{\mathbb{R}^2}u_x(|u_{xyy}|^2+| u_{yy}|^2)~dxdy,\\
			&J_2=-\int_{\mathbb{R}^2}(u_xu_{xyy}+2u_yu_{xxy}+u_{yy}u_{xx})u_{xyy}+(6u_yu_{xy}+2u_xu_{yy})u_{yy}~dxdy.
		\end{align*}
		It follows that
		\begin{align}
			\left|\sum_{i=1}^{4}J_i\right|&\leq C\Bigl(\|\nabla u\|_{L^{\infty}}(\|u_{xyy}\|^2_{L^2}+\|u_{yy}\|^2_{L^2})\nonumber\\
			&+\|\nabla u\|_{L^{\infty}}\|\nabla u_{xx}\|^2_{L^{2}}+\|u_{xx}u_{yy}\|_{L^{2}}\|u_{xyy}\|_{L^2}
			\Bigl) \label{case s=2 fanshuguji-J1+J2+J3+J4(1)}.
		\end{align}
		Thanks to Lemma \ref{Sobolev embedding}, we have
		$$
		\|u_{yy}\|_{L_x^{\infty}L_y^2} \leq C\|u_{yy}\|_{L^2}^{\frac{1}{2}}\|u_{xyy}\|_{L^2}^{\frac{1}{2}}
		\|u_{xx}\|_{L_x^2L_y^{\infty}} \leq C\|u_{xx}\|_{L^2}^{\frac{1}{2}}\|u_{xxy}\|_{L^2}^{\frac{1}{2}},
		$$
		which implies that
		$$
		\|u_{yy}\|_{L_x^{\infty}L_y^2}\|u_{xx}\|_{L_x^2L_y^{\infty}}\|u_{xyy}\|_{L^2} \leq C(\|u_{xx}\|_{L^2} + \|\nabla u_{xx}\|_{L^2})(\|u_{yy}\|_{L^2}^2 + \|u_{xyy}\|_{L^2}^2).
		$$
		Hence, we obtain
		\begin{equation}\label{case s=2 fanshuguji-J1+J2+J3+J4(2)}
			\left|\sum_{i=1}^{4}J_i\right| \leq C\Bigl((\|\nabla u\|_{L^{\infty}} + \|u_{xx}\|_{L^2} + \|\nabla u_{xx}\|_{L^2})(\|u_{xyy}\|_{L^2}^2 + \|u_{yy}\|_{L^2}^2) + \|\nabla u\|_{L^{\infty}}\|\nabla u_{xx}\|_{L^2}^2\Bigr).
		\end{equation}
		Thus, it follows from (\ref{case s=2 fanshuguji-b}) that
		\begin{align}\label{case s=2 fanshuguji-b-jieguo}
			&\frac{d}{dt}(\|u_{yy}\|_{L^2}^2 + \|u_{xyy}\|_{L^2}^2) \leq C(\|\nabla u\|_{L^\infty}\|\nabla u_{xx}\|^2_{L^2}\nonumber\\ 
			& +( \|\nabla u\|_{L^\infty}+\|u_{xx}\|_{L^2}+\|\nabla u_{xx}\|_{L^2})(\|u_{xyy}\|_{L^2}^2 + \|u_{yy}\|_{L^2}^2).
		\end{align}
		Combining (\ref{ZK equation (2')}), (\ref{xian yan gu ji 1}), (\ref{case s=2 fanshuguji-a^*-jieguo}) with (\ref{case s=2 fanshuguji-b-jieguo}), there appears that
		\begin{align}\label{case s=2 fanshugujijieguo}
			&\frac{d}{dt}\bigl(\|(u,u_x)\|_{H^1}^2 + \|(\nabla u_{xx}, u_{yy}, u_{xyy})\|_{L^2}^2 \bigr) \nonumber\\
			&\leq C\bigl(\|(u,\nabla u)\|_{L^\infty} + \|(\nabla u_{xx}, \nabla u_{x}\|_{L^2} \bigr)\bigl( \|u\|_{H^2}^2 + \|(\nabla u_{xx}, u_{xyy})\|_{L^2}^2\bigr).
		\end{align}
		Since $\Vert u \Vert_{X^2}^2 \leq \Vert (u,u_x ) \Vert_{H^1}^2 + \Vert( \nabla u_{xx} , u_{yy},u_{xyy} ) \Vert_{L^2}^2+  \leq 2\Vert u \Vert_{X^2}^2$, due to Gronwall's inequality, (\ref{xian yan gu ji jie guo 1}) and (\ref{case s=2 fanshugujijieguo}), we have
		\begin{equation}\label{xian yan gu ji jie guo 2}
			\sup\limits_{\tau\in \left[0,t\right]}\|u(\tau)\|^2_{X^s}\leq2\|u_0\|^2_{X^s} e^{C_0\int_{0}^{t}\|u(\tau)\|_{X^s}d\tau},\quad s\geq 2.
		\end{equation}
		Taking \(T>0\) such that \(4C_0\|u_0\|_{X^s}T<1\), we get from the bootstrap argument that 
		\begin{equation}\label{xian yan gu ji jie guo }
			\sup\limits_{\tau\in \left[0,t\right]}\|u(\tau)\|_{X^s}\leq2\|u_0\|_{X^s},\quad \forall t\in (0,T].
		\end{equation}
		It then follows that for some positive constant $C$
		\begin{equation}\label{xian yan gu ji jie guo daoshu}
			\sup\limits_{\tau\in \left[0,t\right]}\|u_t(\tau)\|_{X^{s-2}}\leq C\|u_0\|_{X^{s}},\quad \forall t\in (0,T].
		\end{equation}
		With these a priori estimates, we may use the classical Friedrichs regularization method to construct the approximate solutions to (\ref{ZK equation (2)}), and then apply the compactness argumen to get the local existence of a solution $u\in C([0,T];X^s)\cap C^1([0,T];X^{s-2})$ to equation (\ref{ZK equation (2)}) with estimates (\ref{xian yan gu ji jie guo }) and (\ref{xian yan gu ji jie guo daoshu}).\\
		
		Our attention is now turned to uniqueness of the solution. Assume that $u$ and $v$ are two solutions to (\ref{ZK equation (1)}) with the same initial data, and set \(w := u-v\). From (\ref{ZK equation (1)}), we have
		\begin{equation}\label{wei 1 xing 1}
			\left\{ 
			\begin{aligned}
				&w_t-w_{txx}+R_1+R_2+R_3=0,\\
				& w|_{t=0}=0,
			\end{aligned}
			\right.
		\end{equation}
		where $R_1 := \frac{3}{2} \left(  \left(  u + v \right) w \right)_x $, $R_2 :=-\left( \frac{1}{2} \left(  u_x + v_x \right)w_x + u w_{xx} + w v_{xx}\right)_x$, $R_3 :=\kappa w_x +  w_{xyy}$. 
		Multiplying (\ref{wei 1 xing 1}) by \(w\) and taking the $L^2$-inner product, one has
		\begin{equation}\label{wei 1 xing 2}
			\frac{d}{dt} \frac{1}{2} \left(  \Vert w \Vert_{L^2}^2 + \Vert w_x \Vert_{L^2}^2 \right) = -\int_{\mathbb{R}^2} R_1 w~dxdy 
			- \int_{\mathbb{R}^2} R_2 w~dxdy -\int_{\mathbb{R}^2} R_3 w~dxdy.
		\end{equation}
		Integrating by parts yields 
		$\int_{\mathbb{R}^2} R_3 w~dxdy  =0$. 
		Thanks to the H\"{o}lder inequality, we have 
		\begin{equation}\label{wei 1 xing 3a}
			\left| -\int_{\mathbb{R}^2} R_1 w ~dxdy\right|  =  \left|  \frac{3}{2} \int_{\mathbb{R}^2} (u + v ) w w_x ~dxdy \right|  
			\leq   \frac{3}{2} \Vert (u , v) \Vert_{L^{\infty}} \Vert  w \Vert_{L^2} \Vert w_x \Vert_{L^2},
		\end{equation}
		\begin{align}
			\left|  \int_{\mathbb{R}^2} \left(  u_x + v_x \right)  w_x w_x~dxdy \right| \leq \Vert (u_x , v_x) \Vert_{L^{\infty}} \Vert w_x \Vert_{L^2}^2, 
			\label{wei 1 xing 3b}
		\end{align}
		and 
		\begin{align}
			& 	\left| \int_{\mathbb{R}^2} u w_{xx} w_x  + ww_x v_{xx} ~dxdy  \right|  \nonumber \\
			\leq &\left(  \frac{1}{2} \Vert u_x \Vert_{L^{\infty}} \Vert w_x \Vert_{L^2} + \Vert  v_{xx} w \Vert_{L^2}\right)  \Vert  w_x \Vert_{L^2} \nonumber \\
			\leq & C  \left(  \Vert u_x \Vert_{L^{\infty}} \Vert w_x \Vert_{L^2} + \Vert  v_{xx}   \Vert_{L^2}^{\frac{1}{2}} \Vert v_{xxy} \Vert_{L^2}^{\frac{1}{2}} \Vert w \Vert_{L^2}^{\frac{1}{2}} \Vert w_x \Vert_{L^2}^{\frac{1}{2}}\right)  \Vert  w_x \Vert_{L^2} .
			\label{wei 1 xing 3c}
		\end{align}
		From (\ref{wei 1 xing 3b}) and (\ref{wei 1 xing 3c}), we get
		\begin{align}\label{wei 1 xing 4}
			\left|  \int_{\mathbb{R}^2} R_2 w  ~dxdy \right| \leq \frac{1}{2} \left( 2 \Vert (u, v) \Vert_{L^{\infty}}  + \Vert v_{xx} \Vert_{L^2} + \Vert v_{xxy} \Vert_{L^2} \right)  \left( \Vert w \Vert_{L^2}^2 + \Vert w_x\Vert_{L^2}^2  \right) . 
		\end{align}
		Plugging  (\ref{wei 1 xing 3a}) and (\ref{wei 1 xing 4}) into (\ref{wei 1 xing 2}) gives rise to 
		\begin{align}\label{wei 1 xing 5}
			\frac{d}{dt} \left(  \Vert w \Vert_{L^2}^2 + \Vert w_x \Vert_{L^2}^2 \right) \leq C \Vert (u, v ) \Vert_{X^2} \left(  \Vert  w \Vert_{L^2}^2 + \Vert w_x \Vert_{L^2}^2 \right) . 
		\end{align}  
		Then by Gronwall's inequality, we conclude that \(w=u-v=0\). This ends the proof of the uniqueness. Moreover, thanks to Fatou's lemma and (\ref{xian yan gu ji jie guo }), we may verify that the solution $u$ depends continuously on the initial value $u_0 \in X^s $, and this completes the proof of Theorem \ref{local ZK} . 
	\end{proof}
	\begin{proof}[Proof of Theorem \ref{blow up criterion ZK}]
		We will prove the theorem by considering different cases of  the regular index $s$ $(s\geq 2)$.	To this end, we divide the proof into four steps.\\
		\textbf{Step 1: $s=2$.} Applying Gronwall's inequality to (\ref{case s=2 fanshuguji-a-jieguo}) yields that for all $t \in [0,T^*)$,
		\begin{equation}\label{blow up eq-1}
			\sup_{\tau \in [0,t]} \bigl(\|\nabla u(\tau)\|_{L^2}^2 + \|\nabla u_x(\tau)\|_{L^2}^2\bigr) \leq \bigl(\|\nabla u_0\|_{L^2}^2 + \|\nabla u_{0,x}\|_{L^2}^2\bigr) e^{C_0 \int_0^t \|\nabla u\|_{L^\infty}\,d\tau}.
		\end{equation}
		If $\int_0^{T^*} \|\nabla u\|_{L^\infty}\,d\tau$ is finite, then for all $t \in [0,T^*)$,
		\begin{equation}\label{blow up eq-2}
			\|\nabla u(t)\|_{L^2}^2 + \|\nabla u_x(t)\|_{L^2}^2 \leq M_1(T^*) := \bigl(\|\nabla u_0\|_{L^2}^2 + \|\nabla u_{0,x}\|_{L^2}^2\bigr) e^{C_0 \int_0^{T^*} \|\nabla u\|_{L^\infty}\,d\tau}.
		\end{equation}
		On the other hand, thanks to (\ref{case s=2 fanshuguji-a^*-jieguo}), we have
		\[
		\begin{aligned}
			&\frac{d}{dt}\bigl(\|\nabla u\|_{L^2}^2 + \|\nabla u_x\|_{L^2}^2 + \|\nabla u_{xx}\|_{L^2}^2\bigr) \\
			&\quad \leq C\bigl(\|u\|_{L^2} + \|u_x\|_{L^2} + \|\nabla u\|_{L^\infty} + \|\nabla u_x\|_{L^2}\bigr)\bigl(\|\nabla u_{xx}\|_{L^2}^2 + \|\nabla u_x\|_{L^2}^2 + \|\nabla u\|_{L^2}^2\bigr),
		\end{aligned}
		\]
		which along with (\ref{ZK equation (2')}) and (\ref{blow up eq-2}) leads to
		\begin{equation}\label{blow up eq-3}
			\begin{aligned}
				&\frac{d}{dt}\bigl(\|\nabla u\|_{L^2}^2 + \|\nabla u_x\|_{L^2}^2 + \|\nabla u_{xx}\|_{L^2}^2\bigr) \\
				&\quad \leq C_1\bigl(\|u_0\|_{L^2} + \|u_{0,x}\|_{L^2} + M_1(T^*) + \|\nabla u\|_{L^\infty}\bigr)\bigl(\|\nabla u_{xx}\|_{L^2}^2 + \|\nabla u_x\|_{L^2}^2 + \|\nabla u\|_{L^2}^2\bigr).
			\end{aligned}
		\end{equation}
		Applying Gronwall's inequality to (\ref{blow up eq-3}) yields that for all $t \in [0,T^*)$,
		\begin{equation}\label{blow up eq-4}
			\begin{aligned}
				&\sup_{\tau \in [0,t]} \bigl(\|\nabla u(\tau)\|_{L^2}^2 + \|\nabla u_x(\tau)\|_{L^2}^2 + \|\nabla u_{xx}(\tau)\|_{L^2}^2\bigr) \\
				&\quad \leq \bigl(\|\nabla u_0\|_{L^2}^2 + \|\nabla u_{0,x}\|_{L^2}^2 + \|\nabla u_{0,xx}\|_{L^2}^2\bigr) e^{C_1\left(\int_0^t \|\nabla u\|_{L^\infty}\,d\tau + (\|u_0\|_{L^2} + \|u_{0,x}\|_{L^2} + M_1(T^*))t\right)}.
			\end{aligned}
		\end{equation}
		Hence, we get for all $t \in [0,T^*)$,
		\begin{equation}\label{blow up eq-5}
			\| \nabla u(t) \|_{L^2}^2 + \| \nabla u_x(t) \|_{L^2}^2 + \| \nabla u_{xx}(t) \|_{L^2}^2 \leq M_2(T^*), 
		\end{equation}
		where 
		\begin{equation}
			\begin{aligned}
				&M_2(T^*) := \bigl( \| \nabla u_0 \|_{L^2}^2 + \| \nabla u_{0,x} \|_{L^2}^2 + \| \nabla u_{0,xx} \|_{L^2}^2 \bigr) \\\nonumber
				&\quad \times e^{C_1 \Bigl( \int_0^{T^*} \| \nabla u \|_{L^\infty} \,d\tau + \bigl( \| u_0 \|_{L^2}^2 + \| u_{0,x} \|_{L^2}^2 + M_1(T^*) \bigr) T^* \Bigr)}. \nonumber
			\end{aligned}
		\end{equation}
		Similarly, thanks to (\ref{remark-lemma-local well-posedness-3}) and (\ref{case s=2 fanshugujijieguo}), we get
		\begin{align*}
			&\frac{d}{dt} \Bigl( \| (u, u_x) \|_{H^1}^2 + \| (\nabla u_{xx}, u_{yy}, u_{xyy}) \|_{L^2}^2  \Bigr)\\
			&\leq C_2 \Bigl( \| u_0 \|_{L^2} + \| u_{0,x} \|_{L^2}+ \| \nabla u \|_{L^\infty} + \| (\nabla u_x, \nabla u_{xx}) \|_{L^2} \Bigr) \\
			&\quad\quad\quad\times \Bigl( \| (u, u_x) \|_{H^1}^2 + \| (\nabla u_{xx}, u_{yy}, u_{xyy}) \|_{L^2}^2  \Bigr),
		\end{align*}
		which follows from (\ref{blow up eq-5}) that
		\begin{equation}\label{blow up eq-6}
			\begin{aligned}
				&\frac{d}{dt} \Bigl( \| (u, u_x) \|_{H^1}^2 + \| (\nabla u_{xx}, u_{yy}, u_{xyy}) \|_{L^2}^2  \Bigr) \\
				&\leq C_2 \Bigl( \| u_0 \|_{L^2} + \| u_{0,x} \|_{L^2} + M_2(T^*) + \| \nabla u \|_{L^\infty} \Bigr) \\
				&\quad\quad\quad\times \Bigl( \| (u, u_x) \|_{H^1}^2 + \| (\nabla u_{xx}, u_{yy}, u_{xyy}) \|_{L^2}^2 \Bigr).
			\end{aligned}
		\end{equation}
		Applying Gronwall's inequality to (\ref{blow up eq-6}) yields that for all \( t \in [0, T^*) \),
		\begin{equation}\label{blow up eq-7}
			\begin{aligned}
				&\sup_{\tau \in [0, t]} \Bigl( \| (u, u_x)(\tau) \|_{H^1}^2 + \| (\nabla u_{xx}, u_{yy}, u_{xyy})(\tau) \|_{L^2}^2  \Bigr) \\
				&\leq \Bigl( \| (u_0, u_{0,x}) \|_{H^1}^2 + \| (\nabla u_{xx}, u_{yy}, u_{xyy})(0) \|_{L^2}^2 \Bigr) \\
				&\quad\quad\quad\times e^{C_2 \Bigl( \int_0^t \| \nabla u \|_{L^\infty} \,d\tau + \bigl( \| u_0 \|_{L^2} + \| u_{0,x} \|_{L^2} + M_2(T^*) \bigr) t\Bigr)}.
			\end{aligned}
		\end{equation}
		Hence, we get for all \( t \in [0, T^*) \),
		\begin{equation}\label{blow up eq-8}
			\| (u, u_x)(t) \|_{H^1}^2 + \| (\nabla u_{xx}, u_{yy}, u_{xyy})(t) \|_{L^2}^2 \leq M_3(T^*), 
		\end{equation}
		where
		\begin{equation}
			\begin{aligned}
				&M_3(T^*) := \Bigl( \| (u_0, u_{0,x}) \|_{H^1}^2 + \| (\nabla u_{xx}, u_{yy}, u_{xyy})(0) \|_{L^2}^2 \Bigr)\\\nonumber
				&\quad\quad \times e^{C_2 \Bigl( \int_0^{T^*} \| \nabla u \|_{L^\infty} \,d\tau + \bigl( \| u_0 \|_{L^2} + \| u_{0,x} \|_{L^2} + M_2(T^*_{u_0}) \bigr) T^* \Bigr) }.\nonumber
			\end{aligned}
		\end{equation}
		It then follows that
		\begin{equation}\label{blow up eq-9}
			\| u(t) \|_{X^2}^2 \leq 2M_3(T^*), \quad \forall t \in [0, T^*).
		\end{equation}
		This is a contradiction, which completes the proof for $ s = 2$.\\
		\textbf{Step 2: $s=3$.} Differentiating (\ref{case s=2 eq-3}) with respect to $x$, we have
		\begin{align}
			& \partial_t \nabla u_{xxx} + u_{xxxx} \nabla u + 3 u_{xxx} \nabla u_x + 4 u_{xx} \nabla u_{xx} + 3u_x \nabla u_{xxx}+u\nabla u_{xxxx}    \nonumber \\
			&+G_{xx}\ast(\nabla u_{xxyy}+\kappa\nabla u_{xx}+2u\nabla u+u_x\nabla u_x) -2u_{xx}\nabla u-2u\nabla u_{xx}-4u_x\nabla u_x =0 .\label{case s=3 eq-1} 
		\end{align}
		Multiplying (\ref{case s=3 eq-1}) by $\nabla u_{xxx}$ and integrating it over $\mathbb{R}^2$, we get
		\begin{equation}
			\frac{1}{2}\frac{d}{dt}\|\nabla u_{xxx}\|^2_{L^2}=K_1+K_2+K_3+K_4,
		\end{equation}
		where
		\begin{align*}
			&K_1:=-\int_{\mathbb{R}^2}(3u_x \nabla u_{xxx}+u\nabla u_{xxxx}  )\nabla u_{xxx}~dxdy,\\
			&K_2:=-\int_{\mathbb{R}^2}(u_{xxxx} \nabla u + 3 u_{xxx} \nabla u_x + 4 u_{xx} \nabla u_{xx} )\nabla u_{xxx}~dxdy,\\
			&K_3:=-\int_{\mathbb{R}^2}\Bigl(G_{xx}\ast(2u\nabla u+u_x\nabla u_x) -2u_{xx}\nabla u-2u\nabla u_{xx}-4u_x\nabla u_x \Bigr)\nabla u_{xxx}~dxdy,\\
			&K_4:=-\int_{\mathbb{R}^2}\Big(G_{xx}\ast(\nabla u_{xxyy}+\kappa\nabla u_{xx})\Big)\nabla u_{xxx}~dxdy.
		\end{align*}
		Thanks to integration by parts, we get $K_4=0$, and
		\begin{align}
			|K_1|=\left|-\int_{\mathbb{R}^2}\frac{5}{2}u_x(\nabla u_{xxx})^2\,dxdy\right|\leq C\|\nabla u\|_{L^{\infty}}\|\nabla u_{xxx}\|^2_{L^2},
		\end{align}
		\begin{align}
			|K_2|&\leq \|\nabla u\|_{L^{\infty}}\|\nabla u_{xxx}\|^2_{L^2}+3\|u_{xxx}\nabla u_x\|_{L^2}\|\nabla u_{xxx}\|_{L^2}+4\|u_{xx}\nabla u_{xx}\|_{L^2}\|\nabla u_{xxx}\|_{L^2}\nonumber \\
			&\leq \|\nabla u\|_{L^{\infty}}\|\nabla u_{xxx}\|^2_{L^2}+3\|u_{xxx}\|^{\frac{1}{2}}_{L^2}\|u_{xxxy}\|^{\frac{1}{2}}_{L^2}\|\nabla u_x\|^{\frac{1}{2}}_{L^2}\|\nabla u_{xx}\|^{\frac{1}{2}}_{L^2}\|\nabla u_{xxx}\|_{L^2} \nonumber \\
			&\quad\quad\quad+4\|u_{xx}\|^{\frac{1}{2}}_{L^2}\|u_{xxy}\|^{\frac{1}{2}}_{L^2}\|\nabla u_{xx}\|_{L^2}^{\frac{1}{2}}\|\nabla u_{xxx}\|_{L^2}^{\frac{1}{2}}\|\nabla u_{xxx}\|_{L^2}\nonumber\\
			&\leq C(\|\nabla u\|_{L^{\infty}}+\|\nabla u_{xx}\|_{L^2})(\|\nabla u_x\|^2_{L^2}+\|\nabla u_{xxx}\|^2_{L^2}),
		\end{align}
		\begin{align}
			|K_3|&\leq 2\|\nabla u\|_{L^{\infty}}\|u\|_{L^2}\|\nabla u_{xxx}\|_{L^2}+\|u_{x}\|_{L^{\infty}}\|\nabla u_{x}\|_{L^2}\|\nabla u_{xxx}\|_{L^2}
			\nonumber \\
			&+ 2\|\nabla u\|_{L^{\infty}}\| u_{xx}\|_{L^2}\|\nabla u_{xxx}\|_{L^2}+4\|u_{x}\|_{L^{\infty}}\|\nabla u_{x}\|_{L^2}\|\nabla u_{xxx}\|_{L^2}+2\| u_{x}\|_{L^{\infty}}\|\nabla u_{xx}\|^2_{L^2} \nonumber \\
			&\leq C\|\nabla u\|_{L^{\infty}}(\|u\|_{L^2}^2+\|\nabla u_x\|^2_{L^2}+\|\nabla u_{xx}\|^2_{L^2}+\|\nabla u_{xxx}\|^2_{L^2}).
		\end{align}
		So, we obtain
		\begin{equation}\label{case s=3 guji-1}
			\frac{d}{dt}\|\nabla u_{xxx}\|^2_{L^2}\leq (\|\nabla u\|_{L^{\infty}}+\|\nabla u_{xx}\|_{L^2})(\|u\|_{L^2}^2+\|\nabla u_x\|^2_{L^2}+\|\nabla u_{xx}\|^2_{L^2}+\|\nabla u_{xxx}\|^2_{L^2}).
		\end{equation}
		Combining (\ref{case s=2 fanshuguji-a^*-jieguo}) and (\ref{case s=3 guji-1}) , we have
		\begin{align}\label{case s=3 guji-2}
			&\frac{d}{dt}(\|u\|_{L^2}^2+\|\nabla u\|^2_{L^2}+\|\nabla u_x\|^2_{L^2}+\|\nabla u_{xx}\|^2_{L^2}+\|\nabla u_{xxx}\|^2_{L^2}) \nonumber\\
			&\leq C(\|u\|_{L^{\infty}}+\|\nabla u\|_{L^{\infty}}+\|\nabla u_x\|_{L^{2}}+\|\nabla u_{xx}\|_{L^{2}})\nonumber\\
			&\times (\|u\|_{L^2}^2+\|\nabla u\|^2_{L^2}+\|\nabla u_x\|^2_{L^2}+\|\nabla u_{xx}\|^2_{L^2}+\|\nabla u_{xxx}\|^2_{L^2}).
		\end{align}
		Applying Gronwall's inequality to (\ref{case s=3 guji-2}), according to (\ref{remark-lemma-local well-posedness-3}) and (\ref{ZK equation (2')}), one yields that for all $t\in [0,T^*)$,
		\begin{align}
			&\sup\limits_{\tau\in \left[0,t\right]}\Big(\|u\|_{L^2}^2+\|\nabla u\|^2_{L^2}+\|\nabla u_x\|^2_{L^2}+\|\nabla u_{xx}\|^2_{L^2}+\|\nabla u_{xxx}\|^2_{L^2}\Big)\nonumber\\
			&\leq \Bigl(\|u_0\|_{L^2}^2+\|\nabla u_0\|^2_{L^2}+\|\nabla u_{0,x}\|^2_{L^2}+\|\nabla u_{0,xx}\|^2_{L^2}+\|\nabla u_{0,xxx}\|^2_{L^2}\Bigr)\nonumber\\
			&\quad\quad\times e^{C \Bigl( \int_0^t \| \nabla u \|_{L^\infty} \,d\tau + \bigl( \| u_0 \|_{L^2} + \| u_{0,x} \|_{L^2} + M_2(T^*) \bigr) t\Bigr) }.
		\end{align}
		Hence, for all $t\in [0,T^*)$, we have
		\begin{equation}
			\sup\limits_{\tau\in \left[0,t\right]}\Big(\|u\|_{L^2}^2+\|\nabla u\|^2_{L^2}+\|\nabla u_x\|^2_{L^2}+\|\nabla u_{xx}\|^2_{L^2}+\|\nabla u_{xxx}\|^2_{L^2}\Big)\leq M_4(T^*),
		\end{equation}
		where 
		\begin{align*}
			M_4(T^*)&:= \Bigl(\|u_0\|_{L^2}^2+\|\nabla u_0\|^2_{L^2}+\|\nabla u_{0,x}\|^2_{L^2}+\|\nabla u_{0,xx}\|^2_{L^2}+\|\nabla u_{0,xxx}\|^2_{L^2}\Bigr)\\
			&\times e^{C \Bigl( \int_0^t \| \nabla u \|_{L^\infty} \,d\tau + \bigl( \| u_0 \|_{L^2} + \| u_{0,x} \|_{L^2} + M_2(T^*) \bigr) t\Bigr)}.
		\end{align*}
		Then, acting the operator $\nabla$ to (\ref{case s=2 eq-4}) and (\ref{case s=2 eq-5}) respectively gives rise to
		\begin{align}
			& \partial_t\nabla u_{yy}+u_{xyy}\nabla u+u\nabla u_{xyy}+2u_y\nabla u_{xy}+2u_{xy}\nabla u_y+u_x\nabla u_{yy}+u_{yy}\nabla u_x\nonumber\\
			&\quad+G\ast(\nabla u_{xyyyy}+\kappa\nabla u_{xyy})+G_x\ast\nabla(2uu_{yy}+2u^2_y+u_xu_{xyy}+u^2_{xy})=0,\label{case s=3 eq-2}\\ 
			&\partial_t\nabla u_{xyy}+u\nabla u_{xxyy}+u_{xxyy}\nabla u+u_x\nabla u_{xyy}+u_{xyy}\nabla u_x+2u_{xy}\nabla u_{xy}\nonumber \\
			&\quad+2u_y\nabla u_{xxy}+2u_{xxy}\nabla u_y+u_{xx}\nabla u_{yy}+u_{yy}\nabla u_{xx}+G_x\ast(\nabla u_{xyyyy}+\kappa\nabla  u_{xyy})\nonumber \\
			&+G\ast\nabla(2uu_{yy}+2u^2_y+u_xu_{xyy}+u^2_{xy})-2u\nabla u_{yy}-2u_{yy}\nabla u-4u_y\nabla u_y=0. \label{case s=3 eq-3}
		\end{align}
		Taking the $L^2$ inner-products between two equations in (\ref{case s=3 eq-2})-(\ref{case s=3 eq-3}) and $\nabla u_{yy}$, $\nabla u_{xyy}$ respectively,  and summating them, we get
		\begin{equation}\label{case s=3 guji-3}
			\frac{1}{2}\frac{d}{dt}(\|u_{yy}\|^2_{L^2}+\|u_{xyy}\|^2_{L^2})=L_1+L_2+L_3+L_4+L_5,
		\end{equation}
		where
		\begin{align*}
			&L_1:=-\int_{\mathbb{R}^2}(u\nabla u_{xxyy}+u_{xxyy}\nabla u)\nabla u_{xyy}+(u\nabla u_{xyy}+u_{xyy}\nabla u)\nabla u_{yy}~dxdy,\\
			&L_2:=-\int_{\mathbb{R}^2}(u_x\nabla u_{xyy}+u_{xyy}\nabla u_x+2u_{xy}\nabla u_{xy}+2u_y\nabla u_{xxy}+2u_{xxy}\nabla u_y\\
			&\quad\quad\quad+u_{xx}\nabla u_{yy}+u_{yy}\nabla u_{xx}-2u\nabla u_{yy}-2u_{yy}\nabla u-4u_y\nabla u_y)\nabla u_{xyy}~dxdy,\\
			&L_3:=-\int_{\mathbb{R}^2}(2u_y\nabla u_{xy}+2u_{xy}\nabla u_y+u_x\nabla u_{yy}+u_{yy}\nabla u_x)\nabla u_{yy}\,dxdy,\\
			&L_4:=-\int_{\mathbb{R}^2}\Bigl(G_x\ast\nabla(2uu_{yy}+2u^2_y+u_xu_{xyy}+u^2_{xy})\nabla u_{yy}  \\
			&\quad\quad\quad+G\ast\nabla(2u^2_y+2uu_{yy}+u^2_{xy}+u_xu_{xyy})\nabla u_{xyy}\Bigr)~dxdy,\\
			&L_5:=-\int_{\mathbb{R}^2}G\ast(\nabla u_{xyyyy}+\kappa\nabla u_{xyy})\nabla u_{yy}+G_x\ast(\nabla u_{xyyyy}+\kappa\nabla u_{xyy})\nabla u_{xyy}  ~dxdy.
		\end{align*}
		Applying integration by parts yields $L_4=L_5=0$, and thanks to Lemma \ref{Sobolev embedding} again, we have
		\begin{equation}
			|L_1|\leq \|\nabla u\|_{L^{\infty}}(\|\nabla u_{yy}\|^2_{L^2}+\|\nabla u_{xyy}\|^2_{L^2}),\nonumber
		\end{equation}
		\begin{align}
			|L_2|&\leq(\|u_x\|_{L^{\infty}}\|\nabla u_{xyy}\|_{L^{2}}+\|u_{xyy}\nabla u_x\|_{L^{2}}+2\|u_{xy}\nabla u_{xy}\|_{L^{2}}+2\|u_y\|_{L^{\infty}}\|\nabla u_{xxy}\|_{L^{2}}\nonumber\\
			&\quad+2\|u_{xxy}\nabla u_y\|_{L^{2}}+\|u_{xx}\nabla u_{yy}\|_{L^2}+\|u_{yy}\nabla u_{xx}\|_{L^2})\|\nabla u_{yyx}\|_{L^2}\nonumber\\
			&\quad\quad+\|u_x\|_{L^{\infty}}\|\nabla u_{yy}\|_{L^2}^2+2\|\nabla u\|_{L^{\infty}}\|u_{yy}\|_{L^2}\|\nabla u_{yyx}\|_{L^2}+4\|u_y\|_{L^{\infty}}\|\nabla u_y\|_{L^2}\|\nabla u_{yyx}\|_{L^2}\nonumber\\
			&\leq(\|u_x\|_{L^{\infty}}\|\nabla u_{xyy}\|_{L^{2}}+\|\nabla u_x\|_{L^{2}}^{\frac{1}{2}}\|\nabla u_{xy}\|_{L^{2}}^{\frac{1}{2}}\|u_{xyy}\|_{L^{2}}^{\frac{1}{2}}\|u_{xxyy}\|_{L^{2}}^{\frac{1}{2}}\nonumber\\
			&\quad+2\|\nabla u_{xy}\|_{L^{2}}^{\frac{1}{2}}\|\nabla u_{xyy}\|_{L^{2}}^{\frac{1}{2}}\|u_{xy}\|_{L^{2}}^{\frac{1}{2}}\|u_{xxy}\|_{L^{2}}^{\frac{1}{2}}+2\|u_y\|_{L^{\infty}}\|\nabla u_{xxy}\|_{L^{2}}\nonumber\\
			&\quad+2\|\nabla u_y\|^{\frac{1}{2}}_{L^{2}}\|\nabla u_{yy}|^{\frac{1}{2}}_{L^{2}}\|u_{xxy}\|^{\frac{1}{2}}_{L^{2}}\|u_{xxxy}\|^{\frac{1}{2}}_{L^{2}}+\|u_{xx}\|^{\frac{1}{2}}_{L^2}\|u_{xxy}\|^{\frac{1}{2}}_{L^2}\|\nabla u_{yy}\|^{\frac{1}{2}}_{L^2}\|\nabla u_{yyx}\|^{\frac{1}{2}}_{L^2}\nonumber\\
			&\quad+\|u_{yy}\|^{\frac{1}{2}}_{L^2}\|u_{yyx}\|^{\frac{1}{2}}_{L^2}\|\nabla u_{xx}\|^{\frac{1}{2}}_{L^2}\|\nabla u_{xxy}\|^{\frac{1}{2}}_{L^2}     )\|\nabla u_{yyx}\|_{L^2}\nonumber\\
			&\quad\quad+\|u_x\|_{L^{\infty}}\|\nabla u_{yy}\|_{L^2}^2+2\|\nabla u\|_{L^{\infty}}\|u_{yy}\|_{L^2}\|\nabla u_{yyx}\|_{L^2}+4\|u_y\|_{L^{\infty}}\|\nabla u_y\|_{L^2}\|\nabla u_{yyx}\|_{L^2}\nonumber\\
			&\leq C(\|\nabla u\|_{L^{\infty}}+\|\nabla u_x\|_{L^2}+\|\nabla u_{xx}\|_{L^2}+\|u_{yy}\|_{L^2}+\|u_{xyy}\|_{L^2})\times\nonumber\\
			&\quad\quad(\|\nabla u_{xx}\|^2_{L^2}+\|\nabla u_{xyy}\|^2_{L^2}+\|\nabla u_{xxx}\|^2_{L^2}   ),\nonumber
		\end{align}
		and
		\begin{align}
			|L_3|&\leq (2\|u_y\|_{L^{\infty}}\|\nabla u_{xy}\|_{L^2}+2\|u_{xy}\nabla u_y\|_{L^2}+\|u_x\|_{L^{\infty}}\|\nabla u_{yy}\|_{L^2}+\|u_{yy}\nabla u_x\|_{L^2})\|\nabla u_{yy}\|_{L^2}\nonumber\\
			&\leq (2\|u_y\|_{L^{\infty}}\|\nabla u_{xy}\|_{L^2}+2\|u_{xy}\|^{\frac{1}{2}}_{L^2}\|u_{xxy}\|^{\frac{1}{2}}_{L^2}\|\nabla u_y\|^{\frac{1}{2}}_{L^2}\|\nabla u_{yy}\|^{\frac{1}{2}}_{L^2}\nonumber\\
			&\quad\quad+\|u_x\|_{L^{\infty}}\|\nabla u_{yy}\|_{L^2}+\|u_{yy}\|^{\frac{1}{2}}_{L^2}\|u_{yyy}\|^{\frac{1}{2}}_{L^2}\|\nabla u_x\|^{\frac{1}{2}}_{L^2}\|\nabla u_{xx}\|^{\frac{1}{2}}_{L^2}    )\|\nabla u_{yy}\|_{L^2}\nonumber\\
			&\leq C(\|\nabla u\|_{L^{\infty}}+\|u_{yy}\|_{L^2}+\|\nabla u_x\|_{L^2})(\|u_{xyy}\|^2_{L^2}+\|\nabla u_{xx}\|^2_{L^2}+\|\nabla u_{yy}\|^2_{L^2}).\nonumber
		\end{align}
		Hence, we obtain
		\begin{align}\label{case s=3 guji-4}
			\frac{d}{dt}&(\|\nabla u_{yy}\|^2_{L^2}+\|\nabla u_{xyy}\|^2_{L^2})\nonumber\\
			&\leq C(\|\nabla u\|_{L^{\infty}}+\|\nabla u_x\|_{L^2}+\|\nabla u_{xx}\|_{L^2}+\|u_{yy}\|_{L^2}+\|u_{xyy}\|_{L^2})\nonumber\\
			&\quad \times(\|u_{xyy}\|^2_{L^2}+\|\nabla u_{xx}\|^2_{L^2}+\|\nabla u_{yy}\|^2_{L^2}+\|\nabla u_{xxx}\|^2_{L^2}+\|\nabla u_{xyy}\|^2_{L^2}).
		\end{align}
		Combining (\ref{case s=2 fanshuguji-b-jieguo}), (\ref{case s=3 guji-2}) and (\ref{case s=3 guji-4}), one gets
		\begin{align}\label{case s=3 guji-5}
			\frac{d}{dt} &\left( \|(u, u_x)\|_{H^1}^2 + \|(\nabla u_{xx}, \nabla u_{xxx})\|_{L^2}^2 + \|(u_{yy}, u_{xyy}) \|_{H^1}^2  \right)\nonumber\\
			&\leq C(\|u\|_{L^{\infty}}+ \|\nabla u\|_{L^\infty} + \|\nabla u_x\|_{L^2} + \|\nabla u_{xx}\|_{L^2} + \|u_{yy}\|_{L^2} + \|u_{xyy}\|_{L^2} )\nonumber\\
			&\quad \times\left( \|(u, u_x)\|_{H^1}^2 + \|(\nabla u_{xx}, \nabla u_{xxx})\|_{L^2}^2 + \|(u_{yy}, u_{xyy}) \|_{H^1}^2  \right)\nonumber\\
			& \leq C( \|\nabla u\|_{L^\infty} + M_3(T^*) )\left( \|(u, u_x)\|_{H^1}^2 + \|(\nabla u_{xx}, \nabla u_{xxx})\|_{L^2}^2 + \|(u_{yy}, u_{xyy}) \|_{H^1}^2  \right).\nonumber\\
		\end{align}
		Applying Gronwall's inequality to (\ref{case s=3 guji-5}) yields that for all $t\in [0,T^*)$,
		\begin{align}
			&\sup\limits_{\tau\in \left[0,t\right]}\Big(
			\|(u, u_x)\|_{H^1}^2 + \|(\nabla u_{xx}, \nabla u_{xxx})\|_{L^2}^2 + \|(u_{yy}, u_{xyy}) \|_{H^1}^2
			\Big)\nonumber\\
			&\leq \Bigl(
			\|(u_0, u_{0,x})\|_{H^1}^2 + \|(\nabla u_{0,xx}, \nabla u_{0,xxx})\|_{L^2}^2 + \|(u_{0,yy}, u_{0,xyy}) \|_{H^1}^2
			\Bigr)\nonumber\\
			&\quad\quad\times e^{C \Bigl( \int_0^t \| \nabla u \|_{L^\infty} \,d\tau +  M_3(T^*)t\Bigr) }.
		\end{align}
		Thus, for all \( t \in [0, T^*) \), we have
		\begin{equation}\label{blow up eq-10}
			\|(u, u_x)\|_{H^1}^2 + \|(\nabla u_{xx}, \nabla u_{xxx})\|_{L^2}^2 + \|(u_{yy}, u_{xyy}) \|_{H^1}^2  \leq M_5(T^*) ,
		\end{equation}
		where
		\begin{equation}
			\begin{aligned}
				&M_5(T^*) := \Bigl( 
				\|(u_0, u_{0,x})\|_{H^1}^2 + \|(\nabla u_{0,xx}, \nabla u_{0,xxx})\|_{L^2}^2 + \|(u_{0,yy}, u_{0,xyy}) \|_{H^1}^2
				\Bigr)\\\nonumber
				&\quad\quad\times e^{C \Bigl( \int_0^t \| \nabla u \|_{L^\infty} \,d\tau +  M_3(T^*)t\Bigr)}.\nonumber
			\end{aligned}
		\end{equation}
		It then follows that
 		\begin{equation}\label{blow up eq-11}
			\| u(t) \|_{X^3}^2 \leq 3M_5(T^*) ,\quad \forall t \in [0, T^*).
		\end{equation}
		This is a contradiction, which completes the proof of Theorem \ref{blow up criterion ZK} for $s = 3$.\\
		\textbf{Step 3: $s>3$.} In view of (\ref{xian yan gu ji 1}), we have
		\begin{align*}
			\frac{d}{dt} (\|u\|_{H^s}^2 + \|u_x\|_{H^s}^2) &\leq C (\|u\|_{X^2} + \|u\|_{X^3} + \|\nabla u\|_{L^\infty}) (\|u\|_{H^s}^2 + \|u_x\|_{H^s}^2)\\
			&\leq C \left( \sqrt{2M_3(T^*)} + \sqrt{3M_5(T^*)} + \|\nabla u\|_{L^\infty} \right) (\|u\|_{H^s}^2 + \|u_x\|_{H^s}^2).
		\end{align*}
		By Gronwall's inequality, we have for all \( t \in [0, T^*) \)
		\[
		\sup_{\tau \in [0,t]} (\|u\|_{H^s}^2 + \|u_x\|_{H^s}^2) \leq (\|u_0\|_{H^s}^2 + \|u_{0,x}\|_{H^s}^2) e^{C \left( \int_0^t \|\nabla u(\tau)\|_{L^\infty} d\tau + \sqrt{2M_3(T^*)} t + \sqrt{3M_5(T^*)} t \right)}.
		\]
		If \( T^* < \infty \) satisfies \( \int_0^{T^*} \|\nabla u(\tau)\|_{L^\infty} d\tau < \infty \), then we have
		\[
		\|u\|_{H^s}^2 + \|u_x\|_{H^s}^2 \leq M_6(T^*), \quad \forall t \in [0, T^*),
		\]
		where
		\[
		M_6(T^*) := (\|u_0\|_{H^s}^2 + \|u_{0,x}\|_{H^s}^2) e^{C \left( \int_0^{T^*} \|\nabla u(\tau)\|_{L^\infty} d\tau + \sqrt{2M_3(T^*)} T^* + \sqrt{3M_5(T^*)} T^* \right)},
		\]
		which contradicts the assumption that \( T^* < \infty \) is the maximal existence time. This proves Theorem \ref{blow up criterion ZK} for \( s > 3 \).\\
		\textbf{Step 4:  $2 < s < 3$.} By the interpolation inequality, for all $u\in X^3$, we have
		\begin{align*}
			\|u\|_{X^s}^2 &= \|u\|_{H^s}^2 + \|u_x\|_{H^s}^2\\
			&\leq \|u\|_{H^2}^{2(s-2)}\|u\|_{H^3}^{2(3-s)} + \|u_x\|_{H^2}^{2(s-2)}\|u_x\|_{H^3}^{2(3-s)}\\
			&\leq 2\|u\|_{X^2}^{2(s-2)}\|u\|_{X^3}^{2(3-s)}.
		\end{align*}
		If \( \int_0^{T^*} \|\nabla u(\tau)\|_{L^\infty} d\tau < \infty \), from (\ref{blow up eq-9}) and (\ref{blow up eq-11}), we get
		\[
		\|u\|_{X^s}^2 \leq 24 M_3(T^*)^{s-2} M_5(T^*)^{3-s}.
		\]
		Note that $X^3$ is dense in $X^s(2<s<3).$ A standard approximate argument yields
		$$
		\|u\|_{X^s}^2 \leq 25 M_3(T^*)^{s-2} M_5(T^*)^{3-s},~~\forall u\in X^s.
		$$
		This leads to a contradiction, which completes the proof of Theorem \ref{blow up criterion ZK} for \( 2 < s < 3 \). 
		
		Consequently, we have completed the proof of Theorem \ref{blow up criterion ZK} from Step 1 to Step 4.
	\end{proof}
	\section{Blow-up Data}
	We now in this section focus our attention on blow-up data for the CH-ZK equation (\ref{ZK equation (2)}).
	\begin{proof}[Proof of Theorem \ref{blow-up phenomena}]
		In view of the assumption $u(t,x,y)=-u(t,-x,y)$, $\forall t\geq 0,\forall (x,y) \in \mathbb{R}^2$, we have $u(t,0,y)=0$, $(G \ast u) (t,0,y) =0$ and $(G \ast u_{xxyy} )(t,0,y)=0$. Multiplying by $\phi(y)$ on both sides of (\ref{ZK equation (2)}), differentiating it with respect to $x$ and integrating it over $\mathbb{R}$ with respect to $y$, one yields
		\begin{align}
			\int_{\mathbb{R}} \left( u_{tx} + uu_{xx}\right) \phi dy =& \int_{\mathbb{R}} \left( -\frac{1}{2} u_x^2 +  u^2 + \kappa u \right) \phi dy \nonumber \\
			& -\int_{\mathbb{R}} G \ast \left( u^2 +\frac{1}{2} u_x^2 + \kappa u \right) \phi dy -\int_{\mathbb{R}} G \ast u_{xxyy} \phi dy. \nonumber 
		\end{align}
		At the point $(t,0,y)$, we have 
		\begin{align}
			\frac{d}{dt} \int_{\mathbb{R}} u_x (t,0,y) \phi (y) dy = &- \frac{1}{2} \int_{\mathbb{R}} u_x^2(t,0,y) \phi (y) dy -\int_{\mathbb{R}} G \ast \left(  u^2 + \frac{1}{2} u_x^2  \right) (t,0,y) \phi(y) dy . \nonumber 
		\end{align}
		Taking advantage of the H\"{o}lder inequality, one has
		\begin{align}
			\left( \int_{\mathbb{R}} u_x(t,0,y) \phi(y) dy\right)^2 \leq \int_{\mathbb{R}} u_x^2(t,0,y) \phi(y) dy . \nonumber 
		\end{align} 
		Thus, we have 
		\begin{align}
			\frac{d}{dt} \int_{\mathbb{R}} u_x(t,0,y) \phi(y) dy \leq &-\frac{1}{2} \left( \int_{\mathbb{R}} u_x (t,0,y) \phi(y) dy \right)^2.  \nonumber 
		\end{align}
		Let $ B(t) : =\int_{\mathbb{R}} u_x(t,0,y) \phi(y)$,  we have 
		\begin{align}
			\frac{d}{dt}B(t) \leq -\frac{1}{2} B^2(t) . \label{blow-up data eq-1}
		\end{align}
		Since the initial data satisfies $\int_{\mathbb{R}} u_{0,x} (0,y) \phi (y) dy < 0$, we deduce that 
		\begin{align}
			\lim\limits_{t \to T^{\ast}} \int_{\mathbb{R}} u_x(t,0,y) \phi(y) dy =-\infty, \nonumber 
		\end{align}
		with 
		\begin{equation}
			T^{\ast} \leq T_1 := -\frac{2}{\int_{\mathbb{R}} u_{0,x} (0,y)\phi(y)  dy}.\nonumber 
		\end{equation}
		This completes the proof of Theorem \ref{blow-up phenomena}.
	\end{proof}
	\section{Unique continuation property of the solutions} 
	This section is devoted to the study of the unique continuation properity of solutions to the CH-ZK equation (\ref{ZK equation (2)}).
	\begin{proof}[Proof of Theorem \ref{Liuville solution}]
		Integrating over $\mathbb{R}$ by (\ref{ZK equation (2)}) with respect to $y$, there then appears that 
		\begin{equation}\label{Liuville-type 1}
			\int_{\mathbb{R}}u_t+uu_{x}\,dy=\int_{\mathbb{R}}u_t+uu_{x}+G\ast u_{xyy}\,dy=-\int_{\mathbb{R}}G_{x}\ast(u^2+\frac{1}{2}u^2_{x})\,dy.
		\end{equation}
		By virtue of the assumption, we see that
		\begin{equation}\label{Liuville-type 2}
			u^2(t,x,y)+\frac{u^2_{x}(t,x,y)}{2}=0,~~~\forall (t,x)\in \Omega,y\in \mathbb{R}.
		\end{equation}
		From (\ref{Liuville-type 1}), one gets
		\begin{equation}\label{Liuville-type 3}
			\int_{\mathbb{R}}G_{x}\ast(u^2+\frac{1}{2}u^2_{x})~dy\mid_{\Omega}=0.
		\end{equation}
		Since $\Omega$ is an open set, there exists a $t^* \in (0,T)$ and $I = [a,b]$, $a < b$ such that $\{t^*\} \times I \subset \Omega$. Define
		$$
		F(x,y):= G_{x} \ast \left(u^2 + \frac{u_{x}^2}{2}\right)(t^*,x,y) = -\frac{1}{2}\operatorname{sgn}(\cdot) e^{-|\cdot|} * \left(u^2 + \frac{u_{x}^2}{2}\right)(t^*,\cdot,y), 
		$$
		and
		$$
		f(x,y) := \left(u^2 + \frac{u_{x}^2}{2}\right)(t^*,x,y) \geq 0.
		$$
		The regularity of $u$ ensures that $F$ and $f$ are smooth functions. From the above argument, we observe that $\int_{\mathbb{R}} F(x,y)\,dy = 0,~f(x,y) = 0,~\forall x \in [a,b]$. 
		
		On the other hand, for any $z \notin [a,b]$, it holds that $-\operatorname{sgn}(b-z)e^{-|b-z|} > -\operatorname{sgn}(a-z)e^{-|a-z|}$. Hence for any $y \in \mathbb{R}$, it is adduced that
		\begin{equation}\label{Liuville-type 4}
			\begin{aligned}
				F(b,y) &= -\frac{1}{2}\int_{-\infty}^{+\infty} \operatorname{sgn}(b-z)e^{-|b-z|}f(z,y)\,dz \\
				&\geq -\frac{1}{2}\int_{-\infty}^{+\infty} \operatorname{sgn}(a-z)e^{-|a-z|}f(z,y)\,dz = F(a,y).
			\end{aligned}
		\end{equation}
		Since $\int_{\mathbb{R}}(F(b,y) - F(a,y))\,dy = 0$ and $F(b,y) - F(a,y) \geq 0$, it follows that $F(b,y) - F(a,y) = 0$, which implies that $f(x,y) \equiv 0$ in $\mathbb{R}^2$ and also the solution $u(t^*,x,y) = 0$ in $\mathbb{R}^2$. The desired result is thus obtained by the uniqueness of the solution. 
	\end{proof}
	\section{Traveling wave solutions} 
	In this section, we investigate the traveling wave solutions to the CH-ZK equation (\ref{ZK equation (2)}). Let us firstly discuss the existence of peaked solitary-wave solution in the form
	\begin{equation}
		u(t,x,y)=c e^{-\left|x+\beta y-ct\right|},\quad c\in \mathbb{R}.
	\end{equation}
	Such a type of solution is a weak solution in the following sense.
	\begin{definition}
		Given initial data $u_0\in H^1(\mathbb{R}^2)$, the function $u \in C([0,T); H^1_{loc}(\mathbb{R}^2))$ is said to be a weak solution to CH-ZK equation(\ref{ZK equation (2)}), if it satisfies the following identity:
		\begin{align*}
			&\int_0^T \int_{\mathbb{R}^2} \left[ -u\varphi _{t,x}+\varphi _x \left(uu_x+ G_x *(u^2 + \frac{1}{2}u_x^2+\kappa u) \right) -\varphi_{yy}\left(G\ast u\right)_{xx} \right] \,dxdydt \\
			&+ \int_{\mathbb{R}^2} u_0(x,y)\varphi_x(0,x,y)\,dxdy = 0,
		\end{align*}
		for any smooth test function $\varphi(t,x)\in C_c^{\infty}([0,T)\times\mathbb{R}^2)$.If $u$ is a weak solution on $[0,T)$ for every $T>0$, then it is called a global weak solution, where $G(x) = \frac{1}{2}e^{-|x|}$ is the fundamental solution of the operator $(1-\partial_{x}^2)^{-1}\text{ on }\mathbb{R}$, the convolution is taken only with respect to the $x$ variable, and $u_0(x,y) = u(0,x,y)$.
	\end{definition}
	Now we verify the existence of peaked solitary waves and smooth solitary waves, respectively.
	\begin{proof}[Proof of Theorem \ref{xingbojie}]
		 (1) For the existence of peaked solitary waves, assume that
		\begin{equation}\label{global weak peak solution 1}
			u_c(t,x,y) = ce^{-|x +\beta y- ct|}.
		\end{equation}
		Then we get
		\begin{equation}\label{global weak peak solution 2}
			\partial_t u_c = c \cdot \operatorname{sign}(x+\beta y-ct)  u_c, \quad \partial_{x} u_c = -\operatorname{sign}(x+\beta y -ct)  u_c,
		\end{equation}
		which implies $\frac{1}{2}(\partial_{x} u_c)^2 = \frac{1}{2}u_c^2$ and $u_c^2 + \frac{1}{2}(\partial_{x} u_c)^2 = \frac{3}{2}u_c^2$. Notice from (\ref{global weak peak solution 1}) that $G_{x}(x) = -\frac{1}{2}\operatorname{sign}(x)e^{-|x|}$ for $x \in \mathbb{R}$, we have
		\begin{equation}\label{global weak peak solution 3}
			G_{x} * \Bigl(u_c^2 + \frac{1}{2}(\partial_{x} u_c)^2\Bigr)(t,x,y) = \int_{-\infty}^{+\infty} \Bigl(-\frac{1}{2}\operatorname{sign}(x+\beta y-z)e^{-|x+\beta y-z|}\Bigr) \frac{3}{2}c^2 \cdot e^{-2|z-ct|}\,dz.
		\end{equation}
		When $x+\beta y > ct$, we split the right hand side of (\ref{global weak peak solution 3}) into the following three parts:
		\begin{equation}\label{global weak peak solution 4}
			\begin{aligned}
				I &= G_{x} * \Bigl(u_c^2 + \frac{1}{2}(\partial_{x} u_c)^2\Bigr)(t,x,y) \\
				&= \biggl(\int_{-\infty}^{ct} + \int_{ct}^{x+\beta y} + \int_{x+\beta y}^{+\infty}\biggr) \Bigl(-\frac{1}{2}\operatorname{sign}(x+\beta y-z)e^{-|x+\beta y-z|}\Bigr) \frac{3}{2}c^2 \cdot e^{-2|z-ct|}\,dz.\\
				&=: I_1 + I_2 + I_3.
			\end{aligned}
		\end{equation}
		We directly compute $I_1$ as follows:
		\begin{equation}\label{global weak peak solution 5}
			I_1 = -\frac{3}{4}c^2 \int_{-\infty}^{ct} e^{-(x+\beta y-z)} e^{-2(ct-z)}\,dz = -\frac{1}{4}c^2 e^{-(x+\beta y-ct)}.
		\end{equation}
		In a similar manner,
		\begin{equation}\label{global weak peak solution 6}
			I_2 = \frac{3}{4}c^2\Bigl(e^{-2(x+\beta y-ct)} - e^{-(x+\beta y-ct)}\Bigr) \quad \text{and} \quad I_3 = \frac{1}{4}c^2 e^{-2(x+\beta y-ct)}.
		\end{equation}
		Plugging (\ref{global weak peak solution 5}) and (\ref{global weak peak solution 6}) into (\ref{global weak peak solution 4}), we deduce that for $x+\beta y > ct$,
		\begin{equation}\label{global weak peak solution 7}
			G_{x} * \Bigl(u_c^2 + \frac{1}{2}(\partial_{x} u_c)^2\Bigr)(t,x,\beta y) = c^2\Bigl(e^{-2(x+\beta y-ct)} - e^{-(x+\beta y-ct)}\Bigr).
		\end{equation}
		While for the case $x+\beta y \leq ct$, we split the right hand side of (\ref{global weak peak solution 3}) into the following three parts:
		\begin{equation}\label{global weak peak solution 8}
			\begin{aligned}
				II &= G_{x} * \Bigl(u_c^2 + \frac{1}{2}(\partial_{x} u_c)^2\Bigr)(t,x,y) \\
				&= \biggl(\int_{-\infty}^{x+\beta y} + \int_{x+\beta y}^{ct} + \int_{ct}^{+\infty}\biggr) \Bigl(-\frac{1}{2}\operatorname{sign}(x+\beta y-z)e^{-|x+\beta y-z|}\Bigr) \frac{3}{2}c^2 \cdot e^{-2|z-ct|}\,dz.\\
				&=: II_1 + II_2 + II_3.
			\end{aligned}
		\end{equation}
		For $II_1$, a direct computation gives rise to
		\begin{equation}\label{global weak peak solution 9}
			II_1 = -\frac{3}{4}c^2 \int_{-\infty}^{x+\beta y} e^{-(x+\beta y-z)} e^{-2(ct-z)}\,dz = -\frac{1}{4}c^2 e^{2(x+\beta y-ct)}.
		\end{equation}
		Similarly, one gets
		\begin{equation}\label{global weak peak solution 10}
			II_2 = \frac{3}{4}c^2\Bigl(e^{x+\beta y-ct} - e^{2(x+\beta y-ct)}\Bigr) \quad \text{and} \quad II_3 = \frac{1}{4}c^2 e^{x+\beta y-ct}.
		\end{equation}
		Plugging (\ref{global weak peak solution 9}) and (\ref{global weak peak solution 10}) into (\ref{global weak peak solution 8}), we deduce that for $x+\beta y \leq ct$,
		\begin{equation}\label{global weak peak solution 11}
			G_{x} * \Bigl(u_c^2 + \frac{1}{2}(\partial_{x} u_c)^2\Bigr)(t,x,y) = c^2\Bigl(e^{x+\beta y-ct} - e^{2(x+\beta y-ct)}\Bigr),
		\end{equation}
		which along with (\ref{global weak peak solution 7}) yields
		\begin{equation}\label{global weak peak solution 12}
			G_{x} * \Bigl(u_c^2 + \frac{1}{2}(\partial_{x} u_c)^2\Bigr)(t,x,y) = 
			\begin{cases}
				c^2\Bigl(e^{-2(x+\beta y-ct)} - e^{-(x+\beta y-ct)}\Bigr), & \text{for } x+\beta y> ct, \\[6pt]
				c^2\Bigl(e^{x+\beta y-ct} - e^{2(x+\beta y-ct)}\Bigr), & \text{for } x+\beta y \leq ct.
			\end{cases}
		\end{equation}
		
		On the other hand, one can deduce
		\[
		\partial_t u_c + u_c \partial_{x} u_c = \bigl(c - c \cdot e^{-|x+\beta y-ct|}\bigr) \operatorname{sign}(x+\beta y-ct) c \cdot e^{-|x+\beta y-ct|},
		\]
		which implies
		\begin{equation}\label{global weak peak solution 13}
			\partial_t u_c + u_c \partial_{x} u_c(t,x,y) = 
			\begin{cases}
				c^2\Bigl(e^{-(x+\beta y-ct)} - e^{-2(x+\beta y-ct)}\Bigr), & \text{for } \quad x+\beta y > ct, \\[6pt]
				c^2\Bigl(e^{2(x+\beta y-ct)} - e^{x+\beta y-ct}\Bigr), & \text{for } \quad x+\beta y \leq ct.
			\end{cases}
		\end{equation}
		Thanks to (\ref{global weak peak solution 12}) and (\ref{global weak peak solution 13}), we get
		\begin{equation}\label{global weak peak solution 14}
			\partial_t u_c + u_c \partial_{x} u_c + G_{x} * \Bigl(u_c^2 + \frac{1}{2}(\partial_{x} u_c)^2\Bigr) \equiv 0,
		\end{equation}
		in the sense of distribution. Next, we are going to consider
		\begin{equation}\label{global weak peak solution 15}
			G_{x} *u_c (t,x,y)=  \int_{-\infty}^{+\infty} \Bigl(-\frac{1}{2}\operatorname{sign}(x+\beta y-z)e^{-|x+\beta y-z|} \Bigr)ce^{-|z-ct|}\,dz.
		\end{equation}
		When $x+\beta y > ct$, we split the right hand side of (\ref{global weak peak solution 15}) into the following three parts:
		\begin{equation}\label{global weak peak solution 16}
			\begin{aligned}
				III &= G_{x}  \ast u_c(t,x,y) \\
				&= -\frac{c}{2} \biggl(\int_{-\infty}^{ct} + \int_{ct}^{x+\beta y} + \int_{x+\beta y}^{+\infty}\biggr) \operatorname{sign}(x+\beta y-z) e^{-|x+\beta y-z|}\cdotp e^{-|z-ct|}\,dz \\
				&=: III_1 + III_2 + III_3.
			\end{aligned}
		\end{equation}
		We directly compute $III_1$ as follows:
		\begin{equation}\label{global weak peak solution 17}
			III_1 = -\frac{c}{2} \int_{-\infty}^{ct} e^{-(x+\beta y-z)-(ct-z)}\,dz = -\frac{c}{4} e^{-(x+\beta y-ct)}.
		\end{equation}
		In a similar manner,
		\begin{equation}\label{global weak peak solution 18}
			III_2 = -\frac{c}{2} (x+\beta y-ct) e^{-(x+\beta y-ct)} \quad \text{and} \quad III_3 = \frac{c}{4}e^{-(x+\beta y-ct)}.
		\end{equation}
		Plugging (\ref{global weak peak solution 17}) and (\ref{global weak peak solution 18}) into (\ref{global weak peak solution 16}) , we deduce that for $x+\beta y> ct$,
		\begin{equation}\label{global weak peak solution 19}
			G_{x} *u_c(t,x,y) = -\frac{c}{2} (x+\beta y-ct) e^{-(x+\beta y-ct)}.
		\end{equation}
		While for the case $x+\beta y\leq ct$, we split the right hand side of (\ref{global weak peak solution 15}) into the following three parts:
		\begin{equation}\label{global weak peak solution 20}
			\begin{aligned}
				IV &= G_{x} *u_c(t,x,y) \\
				&= -\frac{c}{2} \biggl(\int_{-\infty}^{x+\beta y} + \int_{x+\beta y}^{ct} + \int_{ct}^{+\infty}\biggr) \operatorname{sign}(x+\beta y-z) e^{-|x+\beta y-z|}\cdotp e^{-|z-ct|}\,dz \\
				&=: IV_1 + IV_2 + IV_3.
			\end{aligned}
		\end{equation}
		We directly compute $IV_1$ as follows:
		\begin{equation}\label{global weak peak solution 21}
			IV_1 = -\frac{c}{2} \int_{-\infty}^{x+\beta y} e^{-(x+\beta y-z)-(ct-z)}\,dz = -\frac{c}{4}e^{x+\beta y-ct}.
		\end{equation}
		In a similar manner,
		\begin{equation}\label{global weak peak solution 22}
			IV_2 = -\frac{c}{2}(x+\beta y-ct) e^{x+\beta y-ct} \quad \text{and} \quad IV_3 = \frac{c}{4}e^{x+\beta y-ct}.
		\end{equation}
		Plugging (\ref{global weak peak solution 21}) and (\ref{global weak peak solution 22}) into (\ref{global weak peak solution 20}) , we deduce that for $x+\beta y \leq ct$,
		\begin{equation}\label{global weak peak solution 23}
			G_{x} *u_c(t,x,y) = -\frac{c}{2}(x+\beta y-ct) e^{x+\beta y-ct},
		\end{equation}
		which along with (\ref{global weak peak solution 19}) yislds
		\begin{equation}\label{global weak peak solution 24}
			G_{x} *u_c(t,x,y)=
			\begin{cases}
				-\frac{c}{2} (x+\beta y-ct) e^{-(x+\beta y-ct)}, & \text{for } \quad x+\beta y > ct, \\[6pt]
				-\frac{c}{2} (x+\beta y-ct) e^{x+\beta y-ct}, & \text{for } \quad x+\beta y \leq ct.
			\end{cases}
		\end{equation}
		We know $G_x\ast(\partial^2_yu_c)=\beta^2(G_x\ast u_c)$, which along with (\ref{global weak peak solution 24}) implies
		$$
		G_x\ast(\partial^2_yu_c)+G_x\ast\kappa u_c=(\kappa+\beta^2)G_x\ast u_c=0,
		$$
		if and only if $\kappa+\beta^2=0$ in the sense of distribution.
		
		Therefore, the CH-ZK equation(\ref{ZK equation (2)}) has a global weak solution in the peak form of $u(t,x,y) = ce^{-\left|x+\beta y-ct\right|}$
		for some constant $\beta\in\mathbb{R}$ if and only if holds that $\kappa+\beta^2=0.$ \\
		
		(2) When $\kappa+\beta^2\neq0$, for the existence of smooth solitary waves,  denote the traveling wave variable $\xi :=x+\beta y - ct$, and plug \(u = \phi(\xi)\) into the CH-ZK (\ref{ZK equation (1)}), we get
		\begin{equation}\label{peaked solitary-wave solutions 1}
			(\kappa-c)\phi' + 3\phi\phi' - 2\phi'\phi'' + (c+\beta^2)\phi''' - \phi\phi''' = 0. 
		\end{equation}
		Note that (\ref{peaked solitary-wave solutions 1}) can be written as a total derivative:
		\begin{equation}\label{peaked solitary-wave solutions 2}
			\frac{d}{d\xi}\left[ (\kappa-c)\phi + \frac{3}{2}\phi^2 - \frac{1}{2}(\phi')^2 + (c+\beta^2-\phi)\phi'' \right] = 0.
		\end{equation}
		Integrating (\ref{peaked solitary-wave solutions 2}), and imposing the solitary-wave boundary conditions \(\phi, \phi', \phi'' \to 0\) as \(|\xi| \to \infty\), one infers
		\begin{equation}\label{peaked solitary-wave solutions 4}
			(c+\beta^2-\phi) (\phi')^2 = (c-\kappa)\phi^2 - \phi^3 .
		\end{equation}
		If there exists some $\xi_0$ such that $\phi(\xi_0)=c+\beta^2$, then from (\ref{peaked solitary-wave solutions 4}), we have $\kappa+\beta^2=0$, which contradicts the assumption. Hence, $\phi(\xi)\neq c+\beta^2,\forall \xi\in\mathbb{R}$. So, we have
		\begin{equation}\label{peaked solitary-wave solutions 5}
			(\phi')^2= \phi^2 \frac{c-\kappa-\phi}{c + \beta^2 - \phi} = \phi^2 \frac{c+\beta^2 -(\kappa+\beta^2)-\phi}{c + \beta^2 - \phi}. 
		\end{equation}
		Define
		$$M:=\max\left\{\phi(\xi)\mid\xi\in\mathbb{R}\right\},\quad m:=\min\left\{\phi(\xi)\mid\xi\in\mathbb{R}\right\}.$$
		Let us firstly restrict our attention to $\kappa+\beta^2>0$. Notice that $c-\kappa$ is an extremum of $\phi$, without loss of generality, let $\phi(0)=c-\kappa$. Since $\phi(\xi)\neq c+\beta^2,\forall \xi\in\mathbb{R}$, it follows that either $c+\beta^2>M$ or $c+\beta^2<m$. Now, we assert that $c+\beta^2>M$. Indeed, if $c+\beta^2<m$, then $c-\kappa<c+\beta^2<m$, which is a contradiction. Thus, from (\ref{peaked solitary-wave solutions 5}), we get \(M=c-\kappa=\phi(0)\).
		On the other hand, it is easy to verify that $M=\phi(0)>0$, and then the function $\phi$ has no zeros on $\mathbb{R}$,
		which implies that $\phi(\xi)>0$ for all $\xi\in\mathbb{R}$. 
		Moreover, we can deduce that $\phi'$ is negative
		 when $\xi>0$, and is positive when $\xi<0$.
		We may as well only consider
	the case on the positive real axis.
		By solving (\ref{peaked solitary-wave solutions 5}), one yields
		\begin{equation}\label{peaked solitary-wave solutions 6}
			\frac{d\phi}{d\xi} = - \phi \sqrt{\frac{c-\kappa - \phi}{c + \beta^2 - \phi}}.
		\end{equation}
		Set
		\[
		\psi := \sqrt{\frac{c-\kappa - \phi}{c + \beta^2 - \phi}} \in [0, a), \quad a := \sqrt{\frac{c-\kappa}{c+\beta^2}} < 1.
		\]
		Then we get
		\begin{equation}\label{phi}
			\phi = \frac{c-\kappa - (c+\beta^2)\psi^2}{1 - \psi^2}.
		\end{equation}
		Plugging (\ref{phi}) into (\ref{peaked solitary-wave solutions 6}) and integrating it, one arrives at the following relation
		\begin{equation}\label{peaked solitary-wave solutions 7}
			\xi =  \frac{1}{a}\ln\frac{a+\psi}{a-\psi}-\ln\frac{1+\psi}{1-\psi}. 
		\end{equation}
		The right-hand side of (\ref{peaked solitary-wave solutions 7}) is a strictly increasing smooth function with derivative
		\[
		\frac{d}{d\psi}\left( \frac{1}{a}\ln\frac{a+\psi}{a-\psi}-\ln\frac{1+\psi}{1-\psi} \right) = \frac{2}{\psi^2-1} + \frac{2}{a^2-\psi^2} > 0,\quad\quad \psi\in [0,a).
		\]
		Then its inverse \(\psi = \psi(\xi)\) is smooth, so is 
		 the composite function $\phi(\psi(\xi))$.
		
		The case $\kappa+\beta^2<0$ can be considered similarly, and note that the profile of $\phi$ is now negative. Therefore, we complete the proof of Theorem \ref{xingbojie}.
	\end{proof}
	\begin{proof}[Proof of Theorem \ref{Rigidity of solitary waves}]
		(1) If $c < \min\{0, \kappa\}$, multiplying (\ref{ZK equation of travelling wave solutions-1}) by $y\partial_x^{-1}u_y$ and integrating over \( \mathbb{R}^2 \), we have  
		\begin{equation}\label{ZK equation of travelling wave solutions-2}
			\int_{\mathbb{R}^2}(\kappa-c)yuu_y+cyu_{xx}u_y+\frac{3}{2}yu^2u_y-\frac{1}{2}yu^2_xu_y-yuu_{xx}u_y+yu_yu_{yy}\,dxdy=0.
		\end{equation} 
		It follows from integration by parts that  
		\begin{equation}\label{ZK equation of travelling wave solutions-3}
			\int_{\mathbb{R}^2} (\kappa-c)yuu_y+cyu_{xx}u_y\,dxdy=\frac{1}{2}\int_{\mathbb{R}^2}(c-\kappa)u^2+cu_x^2\,dxdy,		
		\end{equation} 
		\begin{equation}\label{ZK equation of travelling wave solutions-4}
			\int_{\mathbb{R}^2}\frac{3}{2}yu^2u_y\,dxdy=-\frac{1}{2}\int_{\mathbb{R}^2}u^3\,dxdy,
		\end{equation} 
		\begin{align}
			\int_{\mathbb{R}^2}-\frac{1}{2}yu^2_xu_y-yuu_{xx}u_y\,dxdy&=\int_{\mathbb{R}^2}\frac{1}{2}yu^2_xu_y+yuu_xu_{xy}\,dxdy\nonumber\\
			&=-\frac{1}{2}\int_{\mathbb{R}^2}uu^2_x\,dxdy,\label{ZK equation of travelling wave solutions-5}
		\end{align}
		\begin{equation}\label{ZK equation of travelling wave solutions-6}
			\int_{\mathbb{R}^2}yu_yu_{yy}\,dxdy=-\frac{1}{2}\int_{\mathbb{R}^2}u^2_y\,dxdy.
		\end{equation}	
		Combining (\ref{ZK equation of travelling wave solutions-3})-(\ref{ZK equation of travelling wave solutions-6}), one gets
		\begin{equation}\label{ZK equation of travelling wave solutions-7}
			\frac{1}{2}\int_{\mathbb{R}^2}(c-\kappa)u^2+cu_x^2-u^3-uu^2_x-u_y^2 \,dxdy=0.
		\end{equation}	
		Applying the operator \( \partial_x^{-1} \) to (\ref{ZK equation of travelling wave solutions-1}), one infers
		\begin{equation}\label{ZK equation of travelling wave solutions-8}
			( \kappa-c)u + cu_{xx} + \frac{3}{2}u^2 -  \frac{1}{2}u_x^2 - uu_{xx}  + u_{yy} = 0.
		\end{equation}	
		Multiply (\ref{ZK equation of travelling wave solutions-8}) by \( u \) and integrate over \( \mathbb{R}^2 \), we deduce
		\begin{equation}\label{ZK equation of travelling wave solutions-9}
			\int_{\mathbb{R}^2}( \kappa-c)u^2 + cuu_{xx} + \frac{3}{2}u^3 -  \frac{1}{2}uu_x^2 - u^2u_{xx}  + uu_{yy}  \,dxdy = 0.
		\end{equation}	
		Thanks to integration by parts again, we get
		\begin{equation}\label{ZK equation of travelling wave solutions-10}
			\int_{\mathbb{R}^2} u^2u_{xx}\,dxdy = -2 \int_{\mathbb{R}^2} uu_x^2\,dxdy.
		\end{equation}	
		Taking (\ref{ZK equation of travelling wave solutions-10}) into (\ref{ZK equation of travelling wave solutions-9}), one yields
		\begin{equation}\label{ZK equation of travelling wave solutions-11}
			\frac{1}{2}\int_{\mathbb{R}^2} 2 ( c-\kappa)u^2 + 2cu_x^2 - 3u^3 - 3uu_x^2 + 2u^2_y  \,dxdy = 0.
		\end{equation}	
		From (\ref{ZK equation of travelling wave solutions-7}) and (\ref{ZK equation of travelling wave solutions-11}), we obtain
		\begin{equation}\label{ZK equation of travelling wave solutions-12}
			\int_{\mathbb{R}^2} \left(-(c - \kappa)u^2 -cu_x^2 + 5u_y^2\right) dxdy = 0.
		\end{equation}	
		Since \( c < \min\{0, \kappa\} \), it then follows that \( u \equiv 0 \). \\
		
		(2) We focus our attention on  the symmetry of the solution to (\ref{ZK equation of travelling wave solutions-1}) under the condition $c>\max\left\{0,\kappa\right\}$. 
		Let us firstly rewrite (\ref{ZK equation of travelling wave solutions-1}) as follows:
		\begin{equation}\label{symmetry eq-1}
			-cu_x+\frac{1}{2}\partial_x(u^2)+G_x\ast(u^2+\frac{1}{2}u_x^2+\kappa u)+G\ast u_{xyy}=0.
		\end{equation}
		Integrating (\ref{symmetry eq-1}) with respect to $x$, one has
		\begin{equation}\label{symmetry eq-2}
			-cu+\frac{1}{2}u^2+G\ast(u^2+\frac{1}{2}u_x^2+\kappa u)+G\ast u_{yy}=0.
		\end{equation}
		From (\ref{symmetry eq-2}), we get 
		\begin{equation}\label{symmetry u=H*[]}
			u=H\ast[G\ast(u^2+\frac{1}{2}u_x^2)+\frac{1}{2}u^2],
		\end{equation}
		where $\widehat{H}(\xi,\eta)=\frac{1+\xi^2}{c-\kappa+ c\xi^2 + \eta^2}$. Let $u=\phi-\phi_{xx}$. Then (\ref{symmetry u=H*[]}) implies that 
		\begin{equation}\label{symmetry phi=K*[]}
			\phi=K\ast[G\ast((\phi-\phi_{xx})^2+\frac{1}{2}(\phi-\phi_{xx})_x^2)+\frac{1}{2}(\phi-\phi_{xx})^2],
		\end{equation}
		where $\widehat{K}(\xi,\eta)=\frac{1}{c-\kappa+ c\xi^2 + \eta^2}$. We thus obtain the non‑negativity of any solution $\phi$ to (\ref{symmetry phi=K*[]})  under the assumption $c>\max\left\{0,\kappa\right\}$. Using the fact $\frac{1}{A} = \int_0^{+\infty} e^{-At} dt$, one has
		$$\frac{1}{c-\kappa + c \xi^2 + \eta^2} = \int_0^{+\infty} e^{-(c-\kappa)t}\cdot e^{-c\xi^2t} \cdot e^{-\eta^2 t}\,dt.$$
		Thus, we obtain
		\begin{align}
			K = \mathcal{F}^{-1}\left(\frac{1}{c -\kappa + c \xi^2 + \eta^2}\right) &= \int_0^{+\infty} e^{-(c-\kappa)t} \mathcal{F}^{-1}\left(e^{-c\xi^2 t}\right)(x) \mathcal{F}^{-1}\left(e^{-\eta^2 t}\right)(y) \,dt\nonumber\\
			&= \int_0^{+\infty} e^{-(c-\kappa)t} \cdot \frac{e^{-\frac{x^2}{4ct}-\frac{y^2}{4t}}}{4\pi t\sqrt{c}}  \,dt\nonumber\\
			&= \frac{1}{4\pi \sqrt{c}} \int_0^{+\infty} \frac{1}{t} e^{-(c-\kappa)t - \frac{x^2}{4ct} - \frac{y^2}{4t}} \,dt.\label{symmetry integral of K}
		\end{align}
		The integral (\ref{symmetry integral of K}) exists, provided that $c>\max\left\{0,\kappa\right\}$.  By a similar argument as in \cite{Esfahani-Levandosky}, we have $\phi(x,y)=\phi(-x,y)$. Since $u=\phi-\phi_{xx}$, it follows that $u(x,y)=u(-x,y)$, which completes the proof of Theorem \ref{Rigidity of solitary waves}.
	\end{proof}
	\noindent {\bf Acknowledgments.}  
	This work was  partially supported by the National Natural Science Foundation of China under grant 11971188.
	
	
	
	
	
	



\begin{thebibliography}{99}
		
		
		
		
		
		
		
		
		
		
		
		
		
		
		
		
		
		\bibitem{Ablowitz-Clarkson}
		{\small \textsc{Ablowitz, M.J., Clarkson, P.A.,}   Solitons, nonlinear evolution equations and inverse scattering. {\it Cambridge University Press.}, (1991), 70-104. }
		
		\bibitem{Bourgain GFA 1993}
		{\small \textsc{Bourgain, J.,}  On the Cauchy problem for the Kadomtsev-Petviashvili equation. {\it Geom. Funct. Anal.}, {\bf 3} (1993), 315-341. }	
		
		
		\bibitem{Bustamante-Isaza-Mejia JDE 2011}
		{\small \textsc{Bustamante, E., Isaza, P., Mej\'{\i}a, J.,} On the support of solutions to the Zakharov-Kuznetsov equation. {\it J. Differ. Equ.}, {\bf 251} (2011), 2728-2736.}
		
		\bibitem{Bustamante-Isaza-Mejia JFA 2013}
		{\small \textsc{Bustamante, E., Isaza, P., Mej\'{\i}a, J.,} On uniqueness properties of solutions of the Zakharov-Kuznetsov equation. {\it J. Funct. Anal.}, {\bf 264} (2013), 2529-2549.}
		
		
		\bibitem{CH} 
		{\small \textsc {Camassa, R., Holm, D.,}  An integrable shallow water equation with peaked solitons, {\it Phys. Rev. Lett.,}  {\bf 71} (1993), 1661-1664.}	
		
		\bibitem{Constantin-Escher-AsNsP}{\small \textsc{Constantin, A., Escher, J.,}   Global existence and blow-up for a shallow water equation,
			{\it Ann. Sc. Norm. Super. Pisa, Cl. Sci.}, {\bf 26} (1998), 303-328.}	
		
		\bibitem{Constantin-Escher-acta}{\small \textsc{Constantin, A., Escher, J.,}   Wave breaking for nonlinear nonlocal shallow water equations,
			{\it Acta Math.}, {\bf 181} (1998), 229-243.}	
		
		\bibitem{Constantin-Escher-CPAM}{\small \textsc{Constantin, A., Escher, J.,}   Well-posedness, global existence, and blowup phenomena for a periodic quasi-linear hyperbolic equation,
			{\it Comm. Pure Appl. Math.}, {\bf 51} (1998),  475-504.}	
		
		
		\bibitem{Chen-Fan-Wang-Xu MA 2024}
		{\small \textsc{Chen, R.M., Fan, L., Wang, X., Xu, R.,}  Spectral analysis of the periodic $b$-KP equation under transverse perturbations. {\it Math. Ann.}, {\bf 390} (2024), 6315-6354.} 
		
		\bibitem{Dai Act M 1998}
		{\small \textsc{Dai, H.H.,} Model equations for nonlinear dispersive waves in a compressible Mooney–Rivlin rod, {\it Acta Mech.},  {\bf 27} (1998) 193-207.}
		
		
		\bibitem{Esfahani-Levandosky}
		{\small \textsc{Esfahani, A., Levandosky, S.,}  \ Symmetry of the KP-type solitary waves.
			{\it Proc. Amer. Math. Soc.,} {\bf 147}  (2019), 3867-3875.} 
		
		\bibitem{Fuchssteiner-Fokas PD 1981}
		{\small \textsc{Fuchssteiner, B., Fokas, A.S.,} Symplectic structures, their Bäklund transformations and hereditary symmetries, {\it Physica D},  {\bf4} (1981) 47-66.}	
		
		\bibitem{Faminskii 1995}
		{\small \textsc{Faminskii, A.V.,} The Cauchy problem for the Zakharov-Kuznetsov equation. {\it Differ. Equ.}, {\bf 31} (1995), 1002-1012.}
		
		\bibitem{GrunrockHerr 2014}
		{\small \textsc{Gr\"unrock, A., Herr, S.,} The Fourier restriction norm method for the Zakharov-Kuznetsov equation. {\it Disc. Contin. Dyn. Syst. Ser. A.}, {\bf 34} (2014), 2061-2068.}
		
		\bibitem{gui-liu-luo-yin-21}{\small \textsc{Gui, G., Liu, Y., Luo, W., Yin, Z.,}   On a two dimensional nonlocal shallow-water model, {\it  Adv. Math.},
			{\bf392} (2021), 44 pp.}
		
		\bibitem{Geyer-Liu-Pelinovsky JMPA 2024}
		{\small \textsc{Geyer, A., Liu, Y., Pelinovsky, D.E.,}  On the transverse stability of smooth solitary waves in a two-dimensional Camassa-Holm equation. {\it J. Math. Pures Appl.}, {\bf 188} (2024), 1-25.} 
		
		\bibitem{Hadac-Herr-Koch AIHCANL 2009}
		{\small \textsc{Hadac, M., Herr, S., Koch, H.,} Well-posedness and scattering for the KP-II equation in a critical space. {\it Ann. Inst. H. Poincar\'{e} C Anal. Non Lin\'{e}aire}, {\bf 26} (2009) 917-941. }
		
		\bibitem{Korteweg-de Vries 1895}
		{\small \textsc{Korteweg, D.J., Vries, G.de.,}  On the change of form of long waves advancing in a rectangular canal, and on a new type of long stationary waves.  {\it Philos. Mag.}, {\bf 39} (1895), 422-442.} 
		
		\bibitem{KP-70} {\small \textsc{Kadomtsev, B.B., Petviashvili, V.I.,} \ On the stability of solitary waves in weakly dispersing media, {\it Sov. Phys. Dokl.,} {\bf 15} (1970), 539-541.}
		
		\bibitem{Korpusov2014} {\small \textsc{Korpusov, M.O., Sveshnikov, A.G., Yushkov, E.V.,} \ Blow-up of solutions of non-linear equations of Kadomtsev–Petviashvili and Zakharov–Kuznetsov types, {\it Izv. Math.,} {\bf 78} (2014), 500-530.}
		
		\bibitem{KaTo T-Pomce G}
		{\small \textsc{Kato, T., Ponce, G.,}  \ Commutator estimates and the Euler and Navier-Stokes equations.
			{\it Comm. Pure Appl. Amth.,} {\bf 41}  (1988), 891-907.} 
		
		
		\bibitem{Linares-Panthee-Robert-Tzvetkov 2019}
		{\small \textsc{Linares, F., Panthee, M., Robert, T., Tzvetkov, N.,} On the periodic Zakharov-Kuznetsov equation. {\it Discrete Contin. Dyn. Syst.}, {\bf 39} (2019) 3521-3533.}
		
		\bibitem{Linares & Ponce}
		{\small \textsc{Linares, F., Ponce, G.,} Unique continuation properties for solutions to the Camassa-Holm equation and related models. {\it Proc. Amer. Math. Soc.}, {\bf 248} (2020) 3871-3879.}
		
		\bibitem{Moser J}
		{\small \textsc{Moser, J.,}  \ A rapidly convergent iteration method and nonlinear partial differential equations.
			{\it I. Ann. Scuola Norm. Sup. Pisa Cl. Sci. (3),} {\bf 20}  (1966), 265-315.} 
		
		\bibitem{Molinet-Saut-Tzvetkov AIHCANL 2011}
		{\small\textsc{Molinet, L., Saut, J., Tzvetkov, N.,} Global well-posedness for the KP-II equation on the background of a non-localized solution. {\it Ann. Inst. H. Poincar\'{e} C Anal. Non Lin\'{e}aire.}, {\bf 28} (2011) 653-676.  }
		
		
		\bibitem{Molinet-Pilod 2015}
		{\small \textsc{Molinet, L., Pilod, D.,} Bilinear Strichartz estimates for the Zakharov-Kuznetsov equation and applications. {\it Ann. Inst. H. Poincar\'e, Annal. Non.}, {\bf 32} (2015), 347-371.}
		
		\bibitem{Mizumachi MAMS 2015}
		{\small \textsc{Mizumachi, T.,} Stability of line solitons for the KP-II equation in $\mathbb{R}^2$. {\it Mem. Amer. Math. Soc.}, {\bf 238} (2015), vii+95 pp.} 
		
		\bibitem{Panthee 2004}
		{\small \textsc{Panthee, M.,} A note on the unique continuation property for Zakharov-Kuznetsov equation. {\it Nonlinear Anal.}, {\bf 59} (2004), 425-438.}
		
		
		
		\bibitem{Ribaud-Vento 2012 }
		{\small \textsc{Ribaud, F., Vento, S.,} Well-posedness results for the 3D Zakharov-Kuznetsov equation. {\it SIAM J. Math. Anal.}, {\bf 44} (2012), 2289-2304.}
		
		\bibitem{Seadawy 2014}
		{\small \textsc{Seadawy, A.R.,} Stability analysis for two-dimensional ion-acoustic waves in quantum plasmas. {\it Phys. Plasmas.}, {\bf 21} (2014), 9 pp.}
		
		
		\bibitem{Xin-Zhang CPAM 2000}
		{\small \textsc{Xin, Z., Zhang, P.,} On the weak solutions to a shallow water equation,  {\it Comm. Pure Appl. Math.}, {\bf  53} (2000) 1411-1433.}
		
		\bibitem{Zakharov-Kuznetsov JETP 1974}
		{\small \textsc{Zakharov, V.E., Kuznetsov, E.A.,} On three dimensional solutions. {\it Sov. Phys. JETP.}, {\bf 39} (1974) 285-286. }
		
		
		
		
		
		
		
		
		
		
		
	\end{thebibliography}
\end{document}